\documentclass{amsart}
\usepackage{graphicx}
\usepackage[T1]{fontenc}
\usepackage{amsmath,amssymb,amsfonts,amsthm,stmaryrd, mathtools,mathrsfs}
\usepackage{hyperref}
\usepackage{multirow}

\newcommand{\mf}{\mathfrak}
\newcommand{\mc}{\mathcal}
\newcommand{\mb}{\mathbf}

\newcommand{\R}{\mathbf R}
\newcommand{\C}{\mathbf C}
\newcommand{\Q}{\mathbf Q}
\newcommand{\Z}{\mathbf Z}
\newcommand{\F}{\mathbf F}
\newcommand{\N}{\mathbf N}

\newcommand{\Fix}{\textnormal{Fix}}

\newcommand{\norm}[2]{N_{#1}^{#2}}
\newcommand{\cycl}[1]{\Q(\zeta_{#1})}

\numberwithin{equation}{section}
\theoremstyle{plain}
\newtheorem{theorem}{Theorem}

\newtheorem{proposition}[theorem]{Proposition}
\newtheorem{corollary}[theorem]{Corollary}
\theoremstyle{definition}
\newtheorem{definition}[theorem]{Definition}
\newtheorem{example}[theorem]{Example}

\newtheorem{remark}[theorem]{Remark}

\title[Generalized Chebyshev polynomials of the second kind]{Arithmetic exceptionality of generalized Chebyshev polynomials of the second kind}

\author{Derya Acar}
\address{Middle East Technical University, Department of Mathematics, Ankara, Turkey}
\email{adzhar.dariya@metu.edu.tr}
\author{Met\.{I}n Azmaz}
\address{Middle East Technical University, Department of Mathematics, Ankara, Turkey}
\curraddr{\textsc{Gazi University, Department of Mathematics, Ankara, Turkey}}
\email{azmaz.metin@metu.edu.tr}
\email{azmaz.metin@gazi.edu.tr}
\author{Vural Cam}
\address{Selçuk University, Department of Mathematics, Konya, Turkey}
\email{cvural@selcuk.edu.tr}
\author{Ömer Küçüksakallı}
\address{Middle East Technical University, Department of Mathematics, Ankara, Turkey}
\email{komer@metu.edu.tr}

\thanks{This study was supported by The Scientific and Technological Research Council of Türkiye (TÜBİTAK) under Project No. 124F146.}
\date{\today}

\begin{document}
\begin{abstract}
In this paper, we show that generalized Chebyshev polynomials of the second kind, associated with the root system  $A_2$, are not arithmetically exceptional. We achieve this by studying the norms of certain cyclotomic elements that admit a parametrization of finite fields.
\end{abstract}



\maketitle	

\section{Introduction}
Let $p$ be a prime, and let $\F_p$ denote the finite field with $p$ elements. A polynomial
$f\in \mathbf{Z}[\mathbf{x}]$ in $n$ variables is called \emph{arithmetically exceptional} if the induced map
\[f:\F_p^n\to \F_p^n\]
is a permutation for infinitely many primes $p$.

The univariate arithmetically exceptional polynomials are classified: they are precisely the compositions of linear polynomials and Dickson polynomials of the first kind $D_k(x,a)$~\cite{fried1970}. The Dickson polynomials of the first kind are uniquely determined by the functional equation
\begin{equation}\label{eq:dickson1functional}
	D_k\left(y+\frac{a}{y},a\right)=y^k+\frac{a^k}{y^k},
\end{equation}
and admit the explicit formula
\begin{equation}\label{eq:dickson1formula}
	D_k(x,a)
	=
	\sum_{j=0}^{\lfloor k/2\rfloor}
	\frac{k}{k-j}\binom{k-j}{j}
	(-a)^j x^{k-2j}.
\end{equation}
See \cite[(7.5)]{finitefieldsLidl1997} for details of this computation.

The Dickson polynomials of the first kind are closely related to the Chebyshev polynomials $T_k(x)$ of the first kind, defined by $T_k(\cos\theta)=\cos(k\theta)$. Indeed, equation~\eqref{eq:dickson1functional} implies that
\[
D_k(x,1)=2T_k(x/2).
\]

Let $q$ be a power of $p$, and let $\F_q$ denote the finite field with $q$ elements. It is well known that the polynomial $D_k(x,a)$, with $a\neq 0$, permutes the finite field $\F_q$ if and only if $\gcd(q^2-1,k)=1$. Using Dirichlet's theorem on primes in arithmetic progressions, one can deduce that Dickson polynomials of the first kind are arithmetically exceptional for certain values of $k$.

Another important family consists of the Dickson polynomials of the second kind, denoted by $E_k(x,a)$. These are uniquely defined by the functional equation
\begin{equation}\label{eq:dickson2functional}
	E_k\left(y+\frac{a}{y},a\right)
	=
	\frac{y^{k+1}-\dfrac{a^{k+1}}{y^{k+1}}}
	{y-\dfrac{a}{y}}.
\end{equation}
The polynomials $E_k(x,1)$ are closely related to the Chebyshev polynomials $U_k(x)$ of the second kind, defined by $U_k(\cos\theta) \sin\theta = \sin((k+1)\theta)$. Indeeed, equation \eqref{eq:dickson2functional} implies that
\[E_k(x,1)=U_k(x/2)\]

Furthermore, using the identity $D_k'(x,a)=kE_{k-1}(x,a)$, together with~\eqref{eq:dickson1formula}, one derives the explicit formula
\begin{equation*}\label{eq:dickson2formula}
	E_k(x,a)
	=
	\sum_{j=0}^{\lfloor k/2\rfloor}
	\binom{k-j}{j}
	(-a)^j x^{k-2j}.
\end{equation*}

Classifying the integers $k$ for which $E_k(x,1)$ permutes $\F_q$ remains an open problem. Several partial results are known. See \cite{HendersonMatthews, PanarioLima}, and the references therein.

On the other hand, the arithmetic exceptionality of $E_k(x,1)$ is completely understood. One may possibly use Fried's characterization \cite{fried1970} and show that $E_k(x,1)$ cannot be realized as a compositions of linear polynomials and Dickson polynomials of the first kind. However, this approach appears impractical. Instead, we rely on the following theorem.

\begin{theorem}[Cohen~\cite{cohen1994}, Matthews~\cite{matthews1982}]\label{thm:CohenMatt}
	Let $p$ be an odd prime. Then the Dickson polynomial $E_k(x,1)$ of the second kind permutes $\F_p$ if and only if $k$ satisfies the system of congruences
	\[
	\begin{cases}
		k+1\equiv \pm2 \pmod p,\\
		k+1\equiv \pm2 \pmod{\frac{p-1}{2}},\\
		k+1\equiv \pm2 \pmod{\frac{p+1}{2}}.
	\end{cases}
	\]
\end{theorem}

The polynomial $E_1(x,1)=x$ is arithmetically exceptional, trivially. Now let $k\ge 2$ be fixed, and let $p$ be a prime satisfying $p>k+3$. Then
$k+1\not\equiv \pm2 \pmod p$. Hence the first congruence condition of the theorem fails, and therefore $E_k(x,1):\F_p\to\F_p$ is not a permutation whenever $p>k+3$. Consequently, for every $k\ge2$, the polynomial $E_k(x,1)$ is not arithmetically exceptional.

In this paper, we focus on families of polynomial mappings associated with semi-simple Lie algebras that generalize the classical families mentioned above. Indeed, the classical Chebyshev polynomials of the first kind can be realized as a special case of certain polynomial mappings $P_\mathfrak{g}^k:\C^n \rightarrow \C^n$, discovered by Veselov \cite{veselov1987}, and somewhat later by Hoffmann and Withers \cite{hoffmanwithers1988}, following the exponential invariants of Bourbaki \cite{bourbaki2002}. These polynomial mappings $P_\mathfrak{g}^k$ are called the \textit{generalized Chebyshev polynomials} and they have interesting properties. From their definition, they naturally commute
\[P_{\mf{g}}^k \circ P_{\mf{g}}^l = P_{\mf{g}}^l \circ P_{\mf{g}}^k\]
and it is believed that they exhaust all commuting polynomials under certain additional assumptions \cite{veselov-survey1991}. They are orthogonal with respect to a certain measure and can be extended to a complete set of orthogonal polynomials \cite{hoffmanwithers1988}. 

As a basic example, we may consider the family associated with the simple Lie algebra $\mathfrak{g}=A_1$. In this case, we recover the first kind classical polynomials as follows:
\[P_{A_1}^k(x)= D_k(x,1) = 2T_k(x/2).\]
	
The construction of Bourbaki \cite{bourbaki2002}, is very general, and it also allows us to write the classical Chebyshev polynomials of the second kind as maps suitably acting on the associated group ring of exponential invariants. Guided by the Weyl character formula, we consider the \textit{generalized Chebyshev polynomials of the second kind}, denoted $\Q_{\mf{g}}^k$, associated with the semi-simple Lie algebra $\mf{g}$.

In particular, if $\mathfrak{g}=A_1$, then we recover the second kind classical polynomials as follows:
\[Q_{A_1}^k(x)= E_k(x,1) = 2U_k(x/2).\]
To our knowledge, the polynomials $Q^k_{\mf{g}}$ with $\mf{g}\neq A_1$ have not been studied in the theory of finite fields before this work. See \cite[$\S$ 3.3]{hubertsinger2022} for an appearance in numerical analysis.

The polynomial families $P_{A_1}^k$ and $Q_{A_1}^k$ both satisfy the same recurrence relation with different initial conditions \cite{friedlidl1987}. More precisely, we have
\begin{equation}\label{eq:PQA1recursions}
	\begin{gathered}
		P^k_{A_1}=xP^{k-1}_{A_1}-P^{k-2}_{A_1},\text{ with }P^0_{A_1}=2,\ P^1_{A_1}=x,\\
		Q^k_{A_1}=xQ^{k-1}_{A_1}-Q^{k-2}_{A_1},\text{ with }Q^0_{A_1}=1,\ Q^1_{A_1}=x.
	\end{gathered}
\end{equation}
We emphasize that this small change in the initial conditions leads to a completely different behavior with respect to arithmetic exceptionality.

There is a rich literature about the polynomial mappings $P^k_{A_2}(x,y)$. For instance, see \cite{kuc07,lidlwells, uchimura}, and the references therein. We have $P^k_{A_2}=(g_k(y_1,y_2),g_k(y_2,y_1))$ and the polynomials $g_k$ satisfy the recurrence relation
\begin{equation}\label{eq:recrelPA2}
g_k = y_1g_{k-1}-y_2g_{k-2}+g_{k-3} \text{ for }k\ge 3.
\end{equation}
The first few examples of the polynomials $P^k_{A_2}$ are given below:
\begin{align*}
	P_{A_2}^0(y_1,y_2)&=(3,3),\\
	P_{A_2}^1(y_1,y_2)&=(y_1,y_2),\\
	P_{A_2}^2(y_1,y_2)&=(y_1^2-2y_2,y_2^2-2y_1),\\
	P_{A_2}^3(y_1,y_2)&=(y_1^3-3y_1y_2+3,y_2^3-3y_2y_1+3).
\end{align*}
The following result draws an analogy between the polynomial families $P^k_{A_1}$ and $P^k_{A_2}$ in terms of arithmetic exceptionality. 
\begin{theorem}[\cite{lidlwells}]\label{thm:PA2expceptional}
	The bivariate polynomial mappings $P^\ell_{A_2}$ are arithmetically exceptional for prime numbers $\ell \ge 11$.
\end{theorem}

In this paper, we focus on certain bivariate polynomial mappings. These polynomials are given by $Q^k_{A_2}=(h_k(y_1,y_2),h_k(y_2,y_1))$ where the polynomials $h_k$ satisfy the recurrence relation
\begin{equation*}\label{eq:recrelQA2}
	h_k = y_1h_{k-1}-y_2h_{k-2}+h_{k-3} \text{ for }k\ge 3.
\end{equation*}
We emphasize that the recurrence relation proved in Theorem~\ref{thm:A2recursion} coincides with the recurrence relation \eqref{eq:recrelPA2} satisfied by $P_{A_2}^k$. However, the initial conditions differ, leading to distinct families. In fact, the first few instances of the polynomials $Q^k_{A_2}$ are given as follows:
\begin{align*}
	Q_{A_2}^0(y_1,y_2)&=(1,1),\\
	Q_{A_2}^1(y_1,y_2)&=(y_1,y_2),\\
	Q_{A_2}^2(y_1,y_2)&=(y_1^2-y_2,y_2^2-y_1),\\
	Q_{A_2}^3(y_1,y_2)&=(y_1^3-2y_1y_2+1,y_2^3-2y_2y_1+1).
\end{align*}

In this paper, we prove that the bivariate polynomial mappings $Q^k_{A_2}$ are not exceptional for any integer $k \ge 2$. This result draws an analogy between the polynomial families $Q^k_{A_1}$ and $Q^k_{A_2}$ in terms of arithmetic exceptionality. 

The organization of the paper is as follows. In the second section we give the necessary background related to Lie algebras, explaining the polynomial families $P_{\mf{g}}^k$ and $Q_{\mf{g}}^k$. In the third section, we revisit Cohen's theorem and show that the polynomial $Q_{A_1}^k:\F_q \rightarrow \F_q$ does not induce a permutation for all but finitely many primes $p$. In the last section, we prove our main result about $Q_{A_2}^k$ by generalizing the algebraic methods introduced in the third section.

We hope that the ideas of this work can be further generalized to other second kind generalized Chebyshev polynomials associated to simple Lie algebras other than $A_1$ and $A_2$.

\section{Notation and Terminology}
In this section, we provide some terminology and notation that will be used throughout the paper. We mainly follow the content of Bourbaki~\cite[VI, $\S$3]{bourbaki2002} while recalling some notation from \cite{kuc10}.

Let $W$ be the Weyl group associated with a semi-simple Lie algebra $\mf{g}$ of rank $n$. Let $\{\alpha_1, \dots, \alpha_n\}$ be a set of simple roots. The Weyl group $W$ of $\mf{g}$ is a finite reflection group, and each simple root $\alpha_i$ induces a simple reflection in $W$. To each root $\alpha_i$, we associate a coroot
$$\alpha_i^\vee = \frac{2\alpha_i}{\left( \alpha_i, \alpha_i \right)}.$$
The fundamental weights $\{\omega_1, \dots, \omega_n\}$ are then defined by the condition:
\[ \left( \omega_i, \alpha_j^\vee \right) = \delta_{ij} \]
where $\delta_{ij}$ is the Kronecker delta function.

Let $\Lambda$ be the free abelian group generated by the fundamental weights. The elements of $\Lambda$ are called \textit{weights}. Consider the associated group algebra over a unique factorization domain $A$ denoted by $A[\Lambda]$. In this paper, it is sufficient to use $A=\Z$.  A typical element of $A[\Lambda]$ is of the form 
$$x=\sum_{\lambda \in \Lambda}x_\lambda e^\lambda$$ 
with coefficients $x_\lambda\in A$. An element $x \in A[\Lambda]$ is said to be \textit{invariant} under $W$ if
$w(x)= x$ for all $w \in W$. We denote by $A[\Lambda]^W$, the set of $W$-invariant elements of $A[\Lambda]$. 

The Weyl group is generated by reflections, and $\det (w)$ is $\pm 1$ for each element $w\in W$. The following notion is a key property to define the second kind Chebyshev polynomials. 

\begin{definition}
An element $x \in A[\Lambda]$ is said to be \textit{anti-invariant} under $W$ if
\[w(x)= \det (w)\cdot x\]
for all $w \in W$.
\end{definition}

For any $x \in A[\Lambda]$, define 
\[J(x)=\sum_{w \in W} \det (w) \cdot w(x).\]
A straightforward calculation shows that $J(x)$ is anti-invariant. 

There is a one-to-one correspondence between the invariant and anti-invariant submodules of $A[\Lambda]$. To exhibit this correspondence, we let $\rho=\omega_1+ \ldots + \omega_n \in \Lambda$. It turns out that for any $\alpha \in \Lambda$, there exists a $W$-invariant element $\beta \in A[\Lambda]^W$ such that $J(e^{\alpha})=J(e^\rho)\beta$. Moreover, the multiplication by $J(e^\rho)$ is a bijection from $A[\Lambda]^W$ to the set of anti-invariant elements of $A[\Lambda]$.

\begin{definition}
Let $x=\sum_{\lambda \in \Lambda}x_\lambda e^\lambda$ be an element of $A[\Lambda]$. We fix an ordering on weights defined by $ \mu \geq \lambda $ if and only if $\mu - \lambda$ is a nonnegative linear combination of simple roots $\alpha_i$. A term $x_\mu e^\mu$ is called a \textit{maximal term} of $x$ if $\mu \geq \lambda$ for all $\lambda$ with $x_\lambda \neq 0$. 
\end{definition}

\begin{example}
The element $J(e^\rho)$ has $e^\rho$ as its unique maximal term. 	
\end{example}

\begin{example}[$S$-type elements]
For any $\lambda\in\Lambda$, define
$$S(e^\lambda)=\sum_{\mu \in W(\lambda)} e^\mu.$$
Obviously, this is a $W$-invariant element. In particular, we set
$y_i=S(e^{\omega_i})$ for each $1\leq i \leq n$. Each element $y_i$ has $e^{\omega_i}$ as its unique maximal term.	
\end{example}

\begin{example}[$J$-type quotients]
Let $\chi_\lambda$ be the character of an irreducible representation of $\mf{g}$ with the highest weight $\lambda \in \Lambda$. The Weyl Character Formula \cite[$\S$10.4]{hall2015}, states that
\[
\chi_\lambda=J(e^{\rho+\lambda})/J(e^\rho).
\]	
In particular, we set
$z_i=J(e^{\rho+\omega_i})/J(e^\rho)$ for each $1\leq i \leq n$. Each element $z_i$ has $e^{\omega_i}$ as its unique maximal term.
\end{example}

Now we are ready to state the main result of Bourbaki~\cite{bourbaki2002}. This result enables us to introduce interesting families of polynomials, such as $P^k_{\mf{g}}$, and  $Q^k_{\mf{g}}$, generalizing the classical Chebyshev polynomials, $T_k(x)$ and $U_k(x)$, respectively. 

\begin{theorem}\cite[VI, \S3.4, Th.~1]{bourbaki2002}\label{thm:expformche}
Let $\omega_1, \ldots, \omega_n$ be the fundamental weights, and, for $1\leq i \leq n$, let $x_i$ be an element of $A[\Lambda]^W$ with $e^{\omega_i}$ as its unique maximal term. Let $$\varphi:A[X_1,\ldots,X_n] \rightarrow A[\Lambda]^W $$
be the homomorphism from the polynomial algebra $A[X_1,\ldots,X_n]$ to $A[\Lambda]^W$ that takes $X_i$ to $x_i$. Then, the map $\varphi$ is an isomorphism.
\end{theorem}

This theorem has two fundamental consequences, namely the generalized Chebyshev polynomials of the first kind and the second kind. We now give these constructions focusing on the cases associated with $\mf{g}=A_1$, and $\mf{g}=A_2$. 

\subsection{First Kind Generalized Chebyshev Polynomials} These polynomials $P_\mf{g}^k$ were first studied by \cite{hoffmanwithers1988, veselov1987}. With each semi-simple complex Lie algebra $\mf{g}$ of rank $n$, there is an associated infinite sequence of polynomial mappings $P_\mf{g}^k, ~k \in {\N}$ determined from the conditions
\[ \mathbf{y}(k\gamma)=P_\mf{g}^k(\mathbf{y}(\gamma)) \]
where $\mathbf{y}=(y_1,\ldots,y_n)$ is the generalized cosine function. Recall that $y_i=S(e^{\omega_i})$. All coefficients of the polynomials defining $P_\mf{g}^k$ are integers. The formal exponential sums can be considered as complex-valued functions by putting
$$e^\lambda(\gamma) \mapsto e^{-2\pi i(\lambda, \gamma)}$$
as in \cite[Lemma~4.1]{hoffmanwithers1988}.

\begin{example}[$\mf{g}=A_1$]	
In this case, there is a single fundamental weight $\omega=\omega_1$. The Weyl group is given by $W=\left\{\text{id},\sigma\right\}$ where $\sigma(\omega)=-\omega$. We observe that
$$
S(e^{k\omega})=e^{k\omega}+e^{-k\omega},
$$
and this expression can be written as a polynomial of $y_1 = e^\omega+e^{-\omega}$. In other words, $P^k_{A_1}$ is the polynomial that satisfies the functional equation
\begin{equation}\label{eq:PA1functionalformal}
	P^k_{A_1}(e^\omega+e^{-\omega})=e^{k\omega}+e^{-k\omega}.
\end{equation}
We note the similarity between this equation and the equation  (\ref{eq:dickson1functional}) defining the Dickson polynomials of the first kind. We have $P_{A_1}^k(x) = D_k(x,1)$. The first few examples of these polynomials are given by
\begin{align*}
	P^0_{A_1}(x)&=2,\\
	P^1_{A_1}(x)&=x,\\
	P^2_{A_1}(x)&=x^2-2,\\
	P^3_{A_1}(x)&=x^3-3x,\\
	P^4_{A_1}(x)&=x^4-4x^2+2.
\end{align*}
\end{example}

\begin{example}[$\mf{g}=A_2$]\label{ex:PA2}
In this case, the role of the Weyl group is more apparent since there are more symmetries compared to the $A_1$ case. There is some literature about the polynomial mappings $P^k_{A_2}$. See \cite{kuc07,lidlwells, uchimura}, and the references therein for more details. 

Let $W$ be the Weyl group for the Lie algebra $A_2$. The associated Cartan matrix $C$ is defined by $[ (\alpha_i , \alpha_j^\vee )]$. See \cite{humphreys-lie1978} for a complete table of Cartan matrices for simple Lie algebras. For the Lie algebra $A_2$, we have
$$
C=\begin{bmatrix}2&-1\\-1&2\end{bmatrix}.
$$
The Cartan matrix transforms fundamental weights into simple roots. In this case, there are two fundamental weights, namely $\omega_1$, and $\omega_2$. We have
\begin{align*}
	\alpha_1 &= 2\omega_1 - \omega_2,\\
	\alpha_2 &= -\omega_1 +2\omega_2.
\end{align*}
The Weyl group has six elements, and each orbit of a fundamental weight has three elements. We have
\begin{equation}\label{eq:y_1andy_2}
	y_1 =e^{\omega_1}+e^{-\omega_1+\omega_2}+e^{-\omega_2}, 
	\quad y_2 =e^{-\omega_1}+e^{\omega_1-\omega_2}+e^{\omega_2}.
\end{equation}
Recall that $\mathbf{y}=(y_1,y_2)$ is the generalized cosine function associated with $\mf{g}=A_2$. We have $\mathbf{y}(k\gamma)=P_\mf{g}^k(\mathbf{y}(\gamma))$. Let us focus on the case $k=2$. A brief computation reveals that	$$P_{A_2}^2(y_1,y_2)=(y_1^2-2y_2,y_2^2-2y_1).$$
As we have mentioned in the introduction, $P^k_{A_2}=(g_k(y_1,y_2),g_k(y_2,y_1))$ and the polynomials $g_k$ satisfy the recurrence relation \eqref{eq:recrelPA2}. 
This relation may be established by examining the maximal terms of invariant elements having a minimal number of terms. The recurrence satisfied by $P_{A_2}^k$ is an instance of a much broader and richer theory of polynomial sequences defined recursively. We refer the reader to Withers \cite{withers1988} for further details.
\end{example}

\subsection{Second Kind Generalized Chebyshev Polynomials}
Let $\mf{g}$ be an arbitrary semi-simple Lie algebra. For each fundamental weight $\omega_i$, we set $z_i=\chi_{\omega_i}=J(e^{\rho+\omega_i})/J(e^\rho)$. Then, by Theorem~\ref{thm:expformche}, $\chi_{k\omega_i}=J(e^{\rho+k\omega_i})/J(e^\rho)$ can be written as a polynomial of $z_i$'s. The generalized Chebyshev polynomials $Q^k_{\mf{g}}$ of the second kind are the polynomials that satisfy the functional equation
$$Q^k_{\mf{g}}(z_1,\dots,z_n) = \left(\chi_{k\omega_1}, \dots, \chi_{k\omega_n}\right).$$
To the best of our knowledge, the polynomials $Q^k_{\mathfrak g}$, for $\mathfrak g \neq A_1$, have not previously been studied in the context of finite fields. We refer the reader to \cite{hubertsinger2022}, and the references therein for applications of these polynomials in numerical analysis.

\begin{example}[$\mf{g}=A_1$]
In this case, there is a single fundamental weight $\omega=\omega_1$. The Weyl group is given by $W=\left\{\text{id},\sigma\right\}$ where $\sigma(\omega)=-\omega$. We have
$$ \chi_{k\omega} = \frac{J(e^{\omega+k\omega})}{J(e^\omega)} =\frac{e^{(k+1)\omega}-e^{-(k+1)\omega}}{e^\omega-e^{-\omega}}$$
which can be written as a polynomial of
$$
z=\chi_{\omega}=\frac{J(e^{\omega+\omega})}{J(e^\omega)}=\frac{e^{2\omega}-e^{-2\omega}}{e^\omega-e^{-\omega}}=e^\omega+e^{-\omega}.
$$
In other words, $Q^k_{A_1}$ is the polynomial that satisfies the functional equation
\begin{equation}\label{eq:QA1functionalformal}
	Q^k_{A_1}(e^\omega+e^{-\omega})=\frac{e^{(k+1)\omega}-e^{-(k+1)\omega}}{e^\omega-e^{-\omega}}.
\end{equation}
We note the similarity between this equation and the equation  (\ref{eq:dickson2functional}) defining the Dickson polynomials of the second kind. We have $Q_{A_1}^k(x) = E_k(x,1)$. The first few examples of these polynomials are given by
\begin{align*}
	Q^0_{A_1}(x)&=1,\\
	Q^1_{A_1}(x)&=x,\\
	Q^2_{A_1}(x)&=x^2-1,\\
	Q^3_{A_1}(x)&=x^3-2x,\\
	Q^4_{A_1}(x)&=x^4-3x^2+1.
\end{align*}
\end{example}

\begin{example}[$\mf{g}=A_2$]\label{ex:QA2}
Our main result is about the arithmetic exceptionality of the generalized Chebyshev polynomials $Q^k_{A_2}$ of the second kind associated with the Lie algebra $A_2$. 

Recall that $\chi_\lambda$ is the character of an irreducible representation of $A_2$ with highest weight $\lambda$. The polynomial $Q^k_{A_2}$ is defined as the unique bivariate polynomial with integer coefficients that satisfies the following function equation
$$Q^k_{A_2}(z_1,z_2)=(\chi_{k\omega_1},\chi_{k\omega_2}),$$
where $z_1=\chi_{\omega_1}$ and $z_2=\chi_{\omega_2}$. We claim that $z_1=y_1$ and $z_2=y_2$. We have already computed $y_i=S(e^{\omega_i})$ as in (\ref{eq:y_1andy_2}). Now we need to compute $J(e^{\omega_i+\rho})$ for $i=1,2,$ and realize them as a multiple of $J(e^{\rho})$. For this purpose, we may use the following general form of a $J$-type element:
\begin{align}\label{eq:JforA2}
\begin{split}
J(e^{a\omega_1+b\omega_2})=&e^{(-a-b)\omega_1+a\omega_2}+e^{b\omega_1+(-a-b)\omega_2}+e^{a\omega_1+b\omega_2}\\
&-\left(e^{(a+b)\omega_1-b\omega_2}+e^{-a\omega_1+(a+b)\omega_2}+e^{-b\omega_1-a\omega_2}\right).	
\end{split}
\end{align}
Using this formula, it is straightforward to verify that 
$$z_i=\frac{J(e^{\omega_i+\rho})}{J(e^{\rho})} = S(e^{\omega_i})=y_i.$$
We will use the variables $y_i$, as in (\ref{eq:y_1andy_2}), in the rest of the paper, instead of $z_i$. We remark that for other Lie algebras such as $B_2$, and $G_2$, the variables $z_i$ and $y_i$ are not the same for each $i$. 

The components of the pairs $(S(e^{k\omega_1}),S(e^{k\omega_2}))$ and $(\chi_{k\omega_1}, \chi_{k\omega_2})$ are invariant under the action of the Weyl group $W$. Even though they agree when $k=1$, they are distinct for other values of $k$. As a result, $P_{A_2}$ and $Q_{A_2}$ are different families of bivariate polynomial mappings. The first few examples of the polynomials $Q^k_{A_2}$ are as follows:
\begin{align*}
	Q_{A_2}^0(y_1,y_2)&=(1,1),\\
	Q_{A_2}^1(y_1,y_2)&=(y_1,y_2),\\
	Q_{A_2}^2(y_1,y_2)&=(y_1^2-y_2,y_2^2-y_1),\\
	Q_{A_2}^3(y_1,y_2)&=(y_1^3-2y_1y_2+1,y_2^3-2y_2y_1+1),\\
	Q_{A_2}^4(y_1,y_2)&=(y_1^4-3y_1^2y_2+y_2^2+2y_1 , y_2^4-3y_2^2y_1+y_1^2+2y_2).
\end{align*}
We observe that the components of $Q_{A_2}^k$ are symmetric under the change of variables $y_1$ with $y_2$. This is true in general, and it is due to the symmetric nature of $A_2$. To ease the notation in further parts, we write
\begin{equation}\label{eq:QA2symmetry}
Q_{A_2}^k(y_1,y_2)=(h_k(y_1,y_2),h_k(y_2,y_1))
\end{equation}
for some polynomial $h_k(y_1,y_2) \in \Z[\mb{y}]$. It turns out that the polynomials $h_k$ satisfy the recursive relation \eqref{eq:recrelPA2}. More precisely, we have the following theorem. 
\begin{theorem}\label{thm:A2recursion}
We have $h_0=1$, $h_1=y_1$, and $h_2=y_1^2-y_2$. Moreover
$$
h_k=y_1 h_{k-1}-y_2 h_{k-2}+h_{k-3},\text{ for }k\ge 3.
$$
\end{theorem}
\begin{proof}
There are various formulas that convert products of trigonometric functions into sums. For instance, we have $$2\sin(\alpha)2\cos(\beta)=2\sin(\alpha+\beta)+2\sin(\alpha-\beta).$$
Substituting $\alpha=k\theta$ and $\beta=\theta$ into this equation and dividing through by $2\sin(\theta)$, we obtain, after rearranging the terms, the following relation
$$\frac{\sin((k+1)\theta)}{\sin(\theta)} = \cos(\theta)\frac{\sin(k\theta)}{\sin(\theta)} -
\frac{\sin((k-1)\theta)}{\sin(\theta)}. $$
Note that the expressions $2\sin(m\theta)$, with $m=k+1,k,k-1$, are closely related with the formal exponential sums $J(e^{m\omega}) = e^{m\omega}-e^{-m\omega}$. Thus, the above equation proves the recurrence relation for $Q_{A_1}^k$.

We want to generalize this idea to the family $Q_{A_2}^k$. In the $A_2$ case, the $J$-type elements can be computed by (\ref{eq:JforA2}). Note that each such element consists of six exponential terms. However, there may be some cancellations. For instance, if $a=0$ or $b=0$, then $J(e^{a\omega_1+b\omega_2})=0$. For $k\ge 3$, we claim that 
\begin{align}\label{eq:recrelA2proof}
\begin{split}
J(e^{(k+1)\omega_1+\omega_2}) = 
&\ S(e^{\omega_1})J(e^{k\omega_1+\omega_2}) \\
&-S(e^{\omega_2})J(e^{(k-1)\omega_1+\omega_2})\\
&+J(e^{(k-2)\omega_1+\omega_2}).	
\end{split}
\end{align}
This equality can be directly verified by repeatedly applying (\ref{eq:JforA2}), and following the straightforward computations in the group ring $\Z[\Lambda]$. It is easier to perform these computations if one focuses on the maximal terms. For instance $J(e^{(k+1)\omega_1+\omega_2})$ has a unique maximal term $e^{(k+1)\omega_1+\omega_2}$. Such a term occurs on the right-hand side only within the product $S(e^{\omega_1})J(e^{k\omega_1+\omega_2})$, which has $18$ terms. Indeed we have
$$S(e^{\omega_1})J(e^{k\omega_1+\omega_2}) = J(e^{(k+1)\omega_1+\omega_2}) + J(e^{(k-1)\omega_1+2\omega_2}) + J(e^{(k+1)\omega_1}).$$
We observe that the last term on the right-hand side is zero since the coefficient of $\omega_2$ is zero. We also have
$$S(e^{\omega_2})J(e^{(k-1)\omega_1+\omega_2}) = J(e^{(k-1)\omega_1+2\omega_2}) +  J(e^{(k-2)\omega_1+\omega_2})+J(e^{k\omega_1}).$$
Again, the last term on the right-hand side is zero since the coefficient of $\omega_2$ is zero. Using these equalities, we can deduce (\ref{eq:recrelA2proof}). Finally, if we divide each term by $J(e^\rho)=J(e^{\omega_1+\omega_2})$, then we obtain the desired recurrence relation for $h_k$. 
\end{proof}
\end{example}

\subsection{A one-to-one correspondence}
While studying the maps $\F_q^n \rightarrow \F_q^n$, one can use a parametrization of $\F_q^n$ obtained by reducing certain algebraic numbers in a suitable number field modulo a prime ideal of that number field. For this purpose, we recall the main ideas in \cite{kuc10}. The central idea of that paper is the construction of a one-to-one correspondence
\[\F_q^n \leftrightarrow \operatorname{Fix}(P_g^q),\]
between the vector space $\F_q^n$ and the set of complex fixed points of the generalized Chebyshev map $P_\mf{g}^q$. The idea is to prove that there are exactly $q^n$ fixed points of $P_\mf{g}^q$, distinct fixed points remain distinct modulo $p$, and every reduced fixed point lies in $\F_q^n$. This bijection is a key mechanism allowing the dynamics over $\C$ to control permutation behavior over finite fields. Below, we combine several results of \cite{kuc10} into a single statement.
\begin{theorem}\label{thm:FixPtscorrespondence}
Let $\mf{g}$ be a semi-simple Lie algebra of rank $n$, $q$ be a power of a prime $p>n$, $e$ be the exponent of the Weyl group associated with $\mf{g}$, and $\Fix(P_{\mf{g}}^q)$ denote the fixed points of the generalized Chebyshev polynomial $P_{\mf{g}}^q$ of the first kind. Then the number field $\Q(\Fix(P_{\mf{g}}^q))$ is contained in the cyclotomic field $\Q(\zeta_{q^e-1})$. Moreover, for a prime ideal $\mf{p}$ of $\Q(\zeta_{q^e-1})$ lying over $p$, the following reduction map is a bijection:
\begin{align*}
	\Fix(P_{\mf{g}}^q)&\rightarrow \F_q^n\\
	\alpha&\longmapsto \alpha\pmod{\mf{p}}.
\end{align*}
\end{theorem}
This one-to-one correspondence implies that the multivariate polynomial $P_{\mathfrak g}^l$ in $n$ variables is exceptional for every prime $l>e+1$; see \cite[Corollary 4]{kuc10}. In this way, the theory of exceptional polynomial mappings, previously developed for root systems of type $A_n$ by Lidl and Wells \cite{lidlwells}, is extended to the setting of arbitrary simple Lie algebras.

In this paper, our main purpose is to study $Q_{A_2}^k: \F_q^2 \rightarrow \F_q^2$. However, the correspondence $\F_q^2 \leftrightarrow \operatorname{Fix}(P_{A_2}^q)$ is not very useful. Because the dynamics of the first kind and the second kind Chebyshev polynomials are rather different. On the other hand, it is a surprising fact that the correspondence 
$\F_q \leftrightarrow \operatorname{Fix}(P_{A_1}^q)$
is useful to study both families $Q_{A_1}^k$ and $Q_{A_2}^k$. 

The Weyl group of $\mf{g}=A_1$ is isomorphic to $\Z_2$, and its exponent is $e=2$. We want to utilize Theorem~\ref{thm:FixPtscorrespondence}. Let $q$ be a power of a prime $p$ and $\mf{p}$ be a prime ideal of $Q(\zeta_{q^2-1})$ lying over $p$. We can consider formal exponential sums as complex-valued functions by putting $e^\lambda(\gamma)\mapsto e^{-2\pi i(\lambda,\gamma)}$ as in~\cite[Lemma 4.1]{hoffmanwithers1988}. For the Lie algebra $A_1$, there is only one simple root $\alpha$. With $\gamma=u\alpha^\vee$, $u\in\R$, the functional equation \eqref{eq:PA1functionalformal} becomes
$$P_{A_{1}}^{q}(2 \cos(2\pi u)) = 2 \cos(2\pi qu).$$
We are interested in characterizing the fixed points of the map  $P_{A_{1}}^{q}$ by certain cosine values. The cosine function is even and periodic. This implies that $2\pi qu = 2\pi u + 2\pi c$  or $2\pi qu = -2\pi u + 2\pi c$ for some $c \in \mathbf{Z}$. We define the finite sets
$$\mc{A}_{q} = \left\{ 2\cos\left(2\pi \frac{c}{q-1}\right) \mid c \in \mathbf{Z} \right\}, \quad \mc{B}_{q} = \left\{ 2\cos\left(2\pi \frac{c}{q+1}\right) \mid c \in \mathbf{Z} \right\}.$$
These sets consist of algebraic numbers in the cyclotomic extension $\Q(\zeta_{q^2-1})$. We remark that these sets $\mc{A}_{q}$ and $\mc{B}_{q}$ are very similar to the sets $S_1$ and $S_2$ of \cite{cohen1994}, respectively. Throughout the paper, we shall frequently use the bijection
provided by Theorem~\ref{thm:FixPtscorrespondence}. For future reference, we record it as follows
\begin{equation}\label{eq:one-to-one}
	\F_q \leftrightarrow \Fix(P_{A_{1}}^{q}) = \mc{A}_{q} \cup \mc{B}_{q}.
\end{equation}
Considering formal exponential sums as complex-valued functions, the functional equation \eqref{eq:QA1functionalformal} that define the polynomials $Q^k_{A_1}$ becomes
\begin{equation}\label{eq:QA1functionaltrigo}
	\begin{gathered}
		Q^k_{A_1}(2\cos{(2\pi u)})=\frac{\sin{(2\pi(k+1)u)}}{\sin{(2\pi u)}} \text{\quad for } u \in \left(0,\frac{1}{2}\right),\\
		Q^k_{A_1}(2)=k+1,\\
		Q^k_{A_1}(-2)=(-1)^k(k+1).
	\end{gathered}
\end{equation}
We restrict the domain from $\R$ to $[0,1/2]$ to make the related cosine function, namely $2\cos(2\pi u)$, one-to-one. This will be useful while proving Theorem~\ref{thm:invimageofzero}. 

\begin{remark}
Let $q$ be a power of an odd prime $p$. Matthews gives a system of congruences for Dickson polynomials of the second kind to be a permutation of a finite field $\F_q$ \cite[Theorem~2.5]{matthews1982}. More precisely, he preoves that the polynomial $Q^k_{A_1}(x)$ permutes $\F_q$ if $k$ satisfies
\[\begin{cases}
	k+1\equiv \pm2 \pmod p,\\
	k+1\equiv \pm2 \pmod{\frac{q-1}{2}},\\
	k+1\equiv \pm2 \pmod{\frac{q+1}{2}}.
\end{cases}\]
We emphasize that the use of \eqref{eq:one-to-one} in this context is not new. For instance, if $k+1\equiv \pm2 \pmod{\frac{q-1}{2}}$, then $Q^k_{A_1}(x):\mc{A}_{q} \rightarrow \mc{A}_{q}$ is nothing but the map $x\mapsto \pm 1$. A similar conclusion holds for $\mc{B}_{q}$. Matthews employed these ideas in his thesis; however, their origins can be traced to even older results.
\end{remark}

\section{Cohen's Theorem revisited}\label{section:A1}
Recall that Cohen~\cite{cohen1994} gives a system of necessary conditions for the Chebyshev polynomial $Q^k_{A_1}$ of the second kind to be a permutation of $\F_p$, see Theorem~\ref{thm:CohenMatt}. In this section, we will provide a criterion that concludes $Q_{A_1}^k: \F_q \rightarrow \F_q$ is not a permutation if $p$ is sufficiently large compared to a certain function of $k$. If $q=p$, then this result is a partial case of Cohen's result. However, our method can be generalized to the bivariate case.
\begin{theorem}\label{thm:A1main}
Let $k>1$ be an integer. Let $q$ be a power of an odd prime $p$. If $Q_{A_1}^k:\F_q \rightarrow \F_q$ is a permutation, then $p \leq ((k+1)/2)^2+1$. Moreover, $Q_{A_1}^k$ is not arithmetically exceptional.
\end{theorem}

We first observe that the polynomial $Q^k_{A_1}(x)$ is even for even $k$ and is odd for odd $k$. This can be proved directly by the recursive relation~(\ref{eq:PQA1recursions}). If $k$ is even, then $Q^k_{A_1}(x) = h(x^2)$ for some polynomial $h(x)\in\Z[x]$. It follows that the map $Q_{A_{1}}^{k} : \F_{q} \to \F_{q}$ is not a permutation if $k$ is even and $p$ is odd. From this point on, we only focus on the odd values of $k$ and $p$. 

The inverse image of zero under the map $Q_{A_{1}}^{k} : \F_{q} \to \F_{q}$ will play a central role. We adjoin the roots of the polynomial equation $Q_{A_{1}}^{k}(x)=0$ to $\F_q$, and denote the resulting finite field by $\F_{\tilde{q}}$. We will use the one-to-one correspondence $\Fix(P^{\tilde{q}}_{A_1}) \leftrightarrow \F_{\tilde{q}}$ given by \eqref{eq:one-to-one}. Since $k$ is odd, we already have $Q_{A_{1}}^{k}(0)=0$.

Suppose that $p$ divides $(k+1)$ then \eqref{eq:QA1functionaltrigo} implies that $Q^k_{A_1}(2) = Q^k_{A_1}(-2)=0$. In this case, the inverse image of zero under $Q_{A_{1}}^{k} : \F_{q} \to \F_{q}$ is nontrivial since it includes at least three distinct elements, namely $0$, and $\pm 2$. We also note that $2\cos(2\pi u)$ is equal to $2$, and $-2$ for $u=0$, and $u=1/2$, respectively.

We claim that the solutions of the equation $Q_{A_{1}}^{k}(x)\equiv 0 \pmod{\mf{p}}$, with $p \nmid (k+1)$, are in one-to-one correspondence with the following set
$$\mc{Z}_{k} := \left\{ 2\cos\left(2\pi \frac{d}{2(k+1)}\right) \mid c\in \Z \right\} \setminus \{-2,2\}.$$
This set is obviously a subset of $\mc{A}_{{\tilde{q}}} \cup \mc{B}_{{\tilde{q}}}$ by construction. Suppose that $t=e^{2\pi iu}$ for some $u\in \Q \setminus \Z$. In the complex setting, we have the following generic equation by \eqref{eq:QA1functionaltrigo} 
\begin{equation}\label{eq:QA1(alpha)}
	Q_{A_1}^k\left(t+\frac{1}{t} \right)= t^{-k}\left(\dfrac{1-t^{2(k+1)}}{1-t^2}\right) =0.
\end{equation}
The equality on the right-hand side is preserved if the underlying algebraic elements are reduced modulo $\mf{p}$. If $1-t^{2(k+1)}$ is congruent to zero modulo $\mf{p}$, then it must be zero as a complex number as well. This is due to the fact that  $t=e^{2\pi iu}$ is a root of unity for any $u\in\Q$. This proves our claim. 

Now, we focus on the intersection $\mc{Z}_{k} \cap (\mc{A}_{q} \cup \mc{B}_{q})$, and obtain the following theorem.

\begin{theorem}\label{thm:invimageofzero}
Let $k \ge 3$ be an odd integer, $q$ be a power of an odd prime $p$ with $p \nmid (k+1)$. Consider $Q_{A_{1}}^{k} : \F_{q} \to \F_{q}$. We have
$$
|(Q_{A_{1}}^{k})^{-1}(\{0\})| = \gcd\left(\frac{q-1}{2}, k+1\right) + \gcd\left(\frac{q+1}{2}, k+1\right) - 2.
$$
\end{theorem}
\begin{proof} Suppose that $2\cos(2\pi u) \in \mc{A}_{q} \cap \mc{Z}_{k}$. This is possible if and only
\[\frac{c}{q-1} = u = \frac{d}{2(k+1)}\]
for some unique $u \in (0,1/2)$, and integers $c,d$. We set $t = \gcd\left((q-1)/2, k+1\right)$, and write
\[\frac{c}{(q-1)/(2t)} = 2tu = \frac{d}{(k+1)/t}.\]
The denominators, namely  $(q-1)/(2t)$ and $(k+1)/t$, are corpime. Therefore, the above equality holds if and only if $2tu \in \Z$. We have 
$$u \in \left\{ \dfrac{1}{2t}, \dots, \dfrac{t-1}{2t}\right\}.$$ 
Therefore 
$$|\mc{A}_{q} \cap \mc{Z}_{k}| = t-1 = \gcd\left(\frac{q-1}{2}, k+1\right) - 1.$$
A similar conclusion holds for $\mc{B}_{q}$. Note that $\mc{A}_{q} \cap \mc{B}_{q} =\{-2,2\}$. Since $p \nmid (k+1)$, then the inverse image of zero does not contain any of these elements. We have 
\[|(Q_{A_{1}}^{k})^{-1}(\{0\})|= |(\mc{A}_{q} \cup \mc{B}_{q}) \cap \mc{Z}_{k}| = |\mc{A}_{q} \cap \mc{Z}_{k}| + |\mc{B}_{q} \cap \mc{Z}_{k}|.\]
This finishes the proof.
\end{proof}

The inverse image of zero contains a single element if and only if one of the gcd-terms is one, and the other is two. Thus, the following is an immediate consequence of the above theorem.
\begin{corollary}\label{cor:A1invzero}
Let $k \ge 3$ be an odd integer, $q$ be a power of an odd prime $p$. Suppose that $p \nmid (k+1)$. Then, the inverse image of $\{0\}$ under the map $Q_{A_{1}}^{k}:\F_{q} \to \F_{q}$ contains a single element if and only if the following gcd-condition holds 
	$$
	\gcd\left(\frac{q-1}{2}, k+1\right) \cdot \gcd\left(\frac{q+1}{2}, k+1\right) = 2.
	$$
\end{corollary}
We have now seen that, the map $Q^k_{A_1}:\F_q\rightarrow\F_q$ is not a permutation ($k\ge 3$ is odd and $q$ is a power of an odd prime $p$), if $p\mid(k+1)$ or the gcd-condition of the above corollary fails. For the remaining cases, we introduce a method that uses the norms of certain elements from cyclotomic extensions.

Let $K$ be a number field. Recall that the (field) norm of an element $\alpha \in K$ is defined as the product of conjugates of $\alpha$. See \cite{marcus2018} for more details. In general the conjugates may not remain in $K$. However, the cyclotomic extensions are normal extensions, and the conjugate elements are obtained by replacing each root of unity with a suitable power.

Let $\Phi_n(x)$ denote the $n$-th cyclotomic polynomial. Its roots are primitive $n$-th roots of unity. By definition, we have
$$
\Phi_n(x)=\prod_{\substack{1\leq j\leq n\\\gcd(j,n)=1}}\left(x-\zeta_n^j\right).
$$
We now state a well-known fact in the theory of cyclotomic fields. To simplify the notation in further statements, we let $\Z^{+}$ to be the set of positive integers, and let $\mathbb{P}$ be its subset consisting of prime numbers. 
\begin{proposition}\label{prop:norm}
	Let $n$ and $j$ be positive integers such that $\gcd(j,n)=1$. Then,
	$$
	\norm{\Q}{\cycl{n}}\left(1-\zeta_n^j\right)=
	\begin{cases}
		0 & \text{if } n=1, \\
		l & \text{if } n=l^a \text{ for some } l\in\mathbb{P},  a \in \Z^{+}, \\
		1 & \text{otherwise}.
	\end{cases}
	$$
\end{proposition}
\begin{proof} We outline the proof for the convenience of the reader. For $n=1$, the result is obvious. If $n>1$, and $\gcd(j,n)=1$, then 
\[\norm{\Q}{\cycl{n}}\left(1-\zeta_n^j\right) =\Phi_n(1). \]	
This is because $\Phi_n(x+1)$ is the minimal polynomial of the element $1-\zeta_n^j$, and its constant term gives the norm up to a plus or a minus sign. However, the degree of $\Phi_n$ is even, and the minus sign never occurs. There is a natural factorization $(x^n - 1) = \prod_{d|n} \Phi_d(x)$.
Dividing both sides by $x-1=\Phi_1(x)$, and putting $x=1$, we obtain the equality
$$n=\prod_{\substack{d\mid n\\d>1}}\Phi_d(1).$$ 
Let us focus on each factor $\Phi_d(1)$. There are two separate cases. If $d=l^a$ for some $l\in\mathbb{P}$, and $a \in \Z^{+}$, then $\Phi_{l^a}(1)=l$. See \cite[Chapter~2]{marcus2018} for more details. If $d$ has at least two prime factors, then $\Phi_{d}(1)=1$. This is achieved by canceling all the prime factors from $n=\prod\Phi_d(1)$ by eliminating all $\Phi_{l^a}(1)=l$. See \cite[Proposition~2.8]{washingtoncyclotomic1997} for more details.
\end{proof}

We will repeatedly use certain cosine values guided by the one-to-one correspondence \eqref{eq:one-to-one}. For this purpose, we define 
$$\alpha_d:=\zeta_d + \zeta_d^{-1}=2\cos\left(\frac{2\pi}{d}\right).$$
We shall obtain a formula for the norm of $\alpha_d$. We achieve this by first calculating the norm in the extension $\cycl{d}/\Q$ and then recovering the norm in the smaller extension $\Q(\alpha_d)/\Q$. We recall that the field $\Q(\alpha_d)$ is the maximal real subfield of $\cycl{d}$ and we have $[\cycl{d}:\Q(\alpha_d)]=2$ when $d > 2$. Following \cite[Chapter~2]{marcus2018}, we can relate these norms through different extensions, using the relative degree, as follows 
\begin{equation}\label{eq:norminextension}
N_{\Q}^{\cycl{d}}(\alpha_d) = \left( N_{\Q}^{\Q(\alpha_d)}(\alpha_d) \right)^2.	
\end{equation}

The following theorem will be helpful while studying certain coefficients of a polynomial whose roots are conjugates of various elements of the form $\alpha_d$. 

\begin{theorem}\label{thm:A1norminput}
The norm of $\alpha_d$ is given by:
\[N_{\Q}^{\Q(\alpha_d)}(\alpha_d) = 
\begin{cases} 
	2 & \text{if } d = 1, \\
	-2 & \text{if } d = 2, \\
	0 & \text{if } d = 4, \\
	\pm l & \text{if } d = 4l^a\text{ for some }l\in\mathbb{P}, a\in\Z^{+},\\
	\pm 1 & \text{otherwise}.
\end{cases}\]
\end{theorem}
\begin{proof}
The cases where $d$ equals $1,2,$ or $4$ are trivial as $\alpha_1=2$, $\alpha_2=-2$, and $\alpha_4=0$, respectively. For the remaining cases, we shall prove 
$$N_{\Q}^{\cycl{d}}(\alpha_d) = 
\begin{cases}
	l^2 & \text{if } d = 4l^a\text{ for some }l\in\mathbb{P}, a\in\Z^{+},\\
	1 & \text{otherwise}.
\end{cases}
$$
in light of \eqref{eq:norminextension}. We note the identity $$\alpha_d=\zeta_d+\zeta_d^{-1}=\zeta_d^{-1}\left(\frac{1-\zeta_d^4}{1-\zeta_d^2}\right).$$ 
Using the multiplicativity of the norm, and the fact that $\zeta_d^{-1}$ is a unit, we have
$$
N_{\Q}^{\cycl{d}}(\alpha_d)=\pm \frac{N_{\Q}^{\cycl{d}}\left(1-\zeta_d^4\right)}{N_{\Q}^{\cycl{d}}\left(1-\zeta_d^2\right)}.
$$
We want to utilize Proposition~\ref{prop:norm}. In this respect, we analyze the numerator and denominator for all cases of $d$. Starting with non-units, we first consider the cases $d=4l^a$ where $l$ is a prime number and $a$ is a positive integer. We further separate this case by the prime's parity.

If $d=2^a$ with $a\ge 3$, the numerator is $1-\zeta_{2^a}^4 = 1-\zeta_{2^{a-2}}$ and the denominator is $1-\zeta_{2^a}^2 = 1-\zeta_{2^{a-1}}$. Both have norm $2$ in the minimal number fields containing them by Proposition~\ref{prop:norm}. Moreover, the relative degrees of these extension are
$$[\cycl{2^a}:\cycl{2^{a-2}}] = \varphi(2^a)/\varphi(2^{a-2}) = 2^{a-1}/2^{a-3} = 4$$
and
$$[\cycl{2^a}:\cycl{2^{a-1}}] = \varphi(2^a)/\varphi(2^{a-1}) = 2^{a-1}/2^{a-2} = 2.$$
Therefore, we obtain $N_{\Q}^{\cycl{d}}(\alpha_d) = 2^4/2^2 = 2^2$ by using relative norms.
	
If $d=4l^a$ where $l$ is an odd prime and $a\geq1$, the numerator is $1-\zeta_{4l^a}^4 = 1-\zeta_{l^a}$ and the denominator is $1-\zeta_{4l^a}^2 = 1-\zeta_{2l^a}$. By Proposition~\ref{prop:norm}, the numerator has norm $l$ in the field $\cycl{l^a}$. On the other hand, the denominator is a unit in the field $\cycl{2l^a}=\cycl{l^a}$. We have
$$[\cycl{4l^a}:\cycl{l^a}] = \varphi(4l^a)/\varphi(l^a) = \varphi(4) = 2.$$
Therefore, $N_{\Q}^{\cycl{d}}(\alpha_d) = l^2/(\pm 1)^2 = l^2$ by using relative norms.
	
For the remaining cases that yield units, we explain the main ideas and omit the details. If $d=l^a$ or $d=2l^a$ where $l$ is an odd prime, both $1-\zeta_d^4$ and $1-\zeta_d^2$ have the same norm, namely $l$, in the minimal number field containing them, namely $\cycl{l^a}$. If $d=m$ or $d=2m$ where $m$ is an odd integer that is not a prime power, both the numerator and the denominator have norm $1$. This also holds for $d=4m$ where $m$ is not a prime power, regardless of its parity.
\end{proof}

We want to introduce a tool that can be used to prove that the polynomials $Q_{A_1}^k$ with $k>1$ are not exceptional. The finite field $\F_q$ is in one-to-one correspondence with the roots of the polynomial $x^q-x$. We shall check whether the image set $Q_{A_1}^k(\F_q)$ has the same property. By definition, we have 
\begin{equation*}
	\prod\limits_{\alpha\in\Fix(P_{A_1}^q)}(x-\alpha)=P_{A_1}^q(x)-x.    
\end{equation*}
There is a one-to-one correspondence given by \eqref{eq:one-to-one}. Reducing both sides modulo a prime ideal $\mf{p}$ of $\Q(\Fix(P_{A_1}^q)$ lying over $p$, we obtain 
$$\prod\limits_{\alpha\in\F_q}(x-\alpha)\equiv x^q-x\pmod{p}.$$
If we write $P_{A_1}^q(x)-x = x^q+ \ldots+ c_2x^2+ c_1x+ c_0$, then all the coefficients $c_i$ are zero modulo $p$ except the coefficient of $x$, namely $c_1$. Recall that we have $P_{A_1}^k(x) = D_k(x,1)$, and the coefficients of $D_{k}(x, a)$ is given by the formula \eqref{eq:dickson1formula}. Putting $a=1$ in that formula, we find that
\begin{equation}\label{eq:prodfixpoints}
	c_1 = \prod\limits_{\alpha\in\Fix(P_{A_1}^q) \setminus\{0\}} \alpha= (-1)^{(q-1)/2}q-1.
\end{equation}

\begin{remark}\label{rm:coef}
This conclusion on $c_1$ can be alternatively obtained by using Theorem~\ref{thm:A1norminput} up to a plus or a minus sign. Suppose that $q-1$ is divisible by four, and let $d=4l^a$ be a divisor of $q-1$ where $l$ is a prime number. The norm of $\alpha_d$ is $l$ by Theorem~\ref{thm:A1norminput}, and therefore it is a factor of $c_1$. Repeatedly applying Theorem~\ref{thm:A1norminput} for other $\alpha_d$ with nontrivial norm, we precisely recover $\pm(q-1)$. If $q+1$ is divisible by 4, then this process results in $\pm(q+1)$. Note that $c_1=(-1)^{(q-1)/2}q-1$ is always divisible by $4$. This is exactly what one would expect in view of~\ref{thm:A1norminput}.
\end{remark}

Our purpose is to replace $\alpha$ with $Q_{A_1}^k(\alpha_d)$ and analyze the resulting change in the  product, see Theorem~\ref{thm:A1outputprod}. This analysis requires the computation of the norm of the elements $Q_{A_1}^k(\alpha_d)$.

\begin{theorem}\label{thm:A1normQ}
Let $k\ge 3$ be an odd integer and $q$ be a power of an odd prime $p$. Suppose that $p\nmid(k+1)$ and $$\gcd\left(\frac{q-1}{2}, k+1\right) \cdot \gcd\left(\frac{q+1}{2}, k+1\right) = 2.$$
If $d$ is a divisor of $q-1$ or $q+1$, then the norm of $Q_{A_1}^k(\alpha_d)$ is given by:
	$$
	N_{\Q}^{\Q(\alpha_d)}\left(Q_{A_1}^k(\alpha_d)\right) = 
	\begin{cases} 
		k+1 & \text{if } d = 1, \\
		(-1)^k(k+1) & \text{if } d = 2, \\
		0 & \text{if } d = 4, \\
		\pm l & \text{if } d = 4l^a\text{ for some }l\in\mathbb{P}, a\in\Z^{+},\\
		\pm 1 & \text{otherwise}.
	\end{cases}
	$$
\end{theorem}
\begin{proof}
The cases where $d$ equals $1$ or $2$ are obvious by (\ref{eq:QA1functionaltrigo}) as $\alpha_1=2$ and $\alpha_2=-2$, respectively. The case $d=4$ is also trivial as $\alpha_4=0$ is a zero of the polynomial $Q_{A_1}^k$ since $k$ is odd. For the remaining cases, we shall prove 
$$
N_{\Q}^{\Q(\zeta_d)}\left(Q_{A_1}^k(\alpha_d)\right) = 
\begin{cases}
	l^2 & \text{if } d = 4l^a\text{ for some }l\in\mathbb{P}, a\in\Z^{+},\\
	1 & \text{otherwise}.
\end{cases}
$$
in light of (\ref{eq:norminextension}). The functional equation that defines $Q_{A_1}^k$, namely \eqref{eq:QA1functionaltrigo}, allows us to write 
$$Q_{A_1}^k(\alpha_d)=\dfrac{\zeta_d^{k+1}-\zeta_d^{-(k+1)}}{\zeta_d-\zeta_d^{-1}}=\zeta_d^{-k}\left(\dfrac{1-\zeta_d^{2(k+1)}}{1-\zeta_d^2}\right).
$$
Using the multiplicativity of the norm and the fact that $\zeta_d^{-k}$ is a unit, we obtain the following equality
$$N_{\Q}^{\cycl{d}}(Q_{A_1}^k(\alpha_d))=\pm \frac{N_{\Q}^{\cycl{d}}\left(1-\zeta_d^{2(k+1)}\right)}{N_{\Q}^{\cycl{d}}\left(1-\zeta_d^2\right)}.
$$
We want to utilize Proposition~\ref{prop:norm}. We rely on the gcd-condition in the hypothesis to restrict our attention to rather limited cases. Depending on $q$ modulo $4$, the gcd-condition becomes one of the following
	$$
	\begin{cases}
		\gcd\left(\frac{q-1}{2}, k+1\right) = 2 \text{ and } \gcd\left(\frac{q+1}{2}, k+1\right) = 1 &\text{if } q\equiv1\pmod4,\\
		\gcd\left(\frac{q-1}{2}, k+1\right) = 1 \text{ and } \gcd\left(\frac{q+1}{2}, k+1\right) = 2 &\text{if } q\equiv3\pmod4.
	\end{cases}
	$$
As in the hypothesis, suppose that $d$ is a divisor of $q-1$ or $q+1$. We remark that this choice is compatible with the one-to-one correspondence \eqref{eq:one-to-one}, and we will eventually focus on $\alpha_d\in \Fix(P_{A_1}^q)$. 

We repeat the ideas in the proof of Theorem~\ref{thm:A1norminput}, emphasizing the differences between the expressions
\[ \frac{1-\zeta_d^4}{1-\zeta_d^2}, \quad \text{ and } \quad \dfrac{1-\zeta_d^{2(k+1)}}{1-\zeta_d^2}. \]

Suppose that $d=2^a$ with $a\ge 3$. Then the numerator comes with the power $2(k+1)$ instead of $4$. However, the gcd-condition implies that $1-\zeta_{2^a}^{2(k+1)}$ and  $1-\zeta_{2^{a-2}}$ are conjugates of each other. The rest of the proof is the same.  

Suppose that $d=4l^a$ for some odd prime $l$ with $a \ge 1$. Then the numerator is $1-\zeta_{4l^a}^{2(k+1)} = 1-\zeta_{l^a}^{(k+1)/2}$. Using Proposition~\ref{prop:norm} and the gcd-condition, we conclude that the numerator has norm $\pm l$ in the extension $\cycl{l^a}/\Q$. The rest of the proof is the same.

If $d>1$ is odd, then as $d$ is a divisor of $(q-1)/2$ or $(q+1)/2$, and the gcd-condition ensures that $d$ and $k+1$ are coprime. So, the numerator $1-\zeta_d^{2(k+1)}$ and the denominator $1-\zeta_d^2$ have the same norm as $1-\zeta_d$ by Proposition~\ref{prop:norm}.
	
If $d=2m$ with $m>1$ odd, then as $m$ is a divisor of $(q-1)/2$ or $(q+1)/2$, and the gcd-condition ensures that $m$ and $k+1$ are coprime. So, the numerator $1-\zeta_d^{2(k+1)}=1-\zeta_m^{k+1}$ and and the denominator $1-\zeta_d^4=1-\zeta_m^2$ have the same norm as of $1-\zeta_m$ by Proposition~\ref{prop:norm}.
\end{proof}

The following theorem is the last step before we prove our main result in this section. 

\begin{theorem}\label{thm:A1outputprod}
Let $k\geq 3$ be an odd integer and $q$ be a power of an odd prime $p$. Suppose that $p\nmid(k+1)$, and 
$$\gcd\left(\frac{q-1}{2}, k+1\right) \cdot \gcd\left(\frac{q+1}{2}, k+1\right) = 2.$$
Then, we have
$$
\prod\limits_{\alpha\in\Fix(P_{A_1}^q) \setminus\{0\}} Q_{A_{1}}^{k}(\alpha) = \pm\left(\frac{k+1}{2}\right)^{2} \left((-1)^{(q-1)/2}q-1\right).$$
\end{theorem}
\begin{proof}
Any element of $\Fix(P_{A_1}^q)\setminus\{0\}$ is a conjugate of $\alpha_d$ for some divisor $d$ of $q-1$ or $q+1$. Moreover, all conjugates of $\alpha_d$ are included in the number field $\Q(\alpha_d)$. This is a direct consequence of the one-to-one correspondence given by \eqref{eq:one-to-one}. 

Being a conjugate is an equivalence relation. This gives a partition of $\Fix(P_{A_1}^q)$, and we choose a natural representative for each equivalence class, namely $\alpha_d$. Let $S$ be the subset of positive integers $1\leq d \leq q+1$ such that $\alpha_d \in \Fix(P_{A_1}^q)$. The set $S$ consists of divisors of $q-1$ or $q+1$ by \eqref{eq:one-to-one}. We have the following partition induced by the conjugacy equivalence relation 
$$ \Fix(P_{A_1}^q) = \bigcup_{d\in S} [\alpha_d]$$
The equivalence class of $\alpha_4=0$ has a single element since $\alpha_4\in\Q$. Symbolically, we have $[\alpha_4] = \{0\}$. It follows that 
$$\prod\limits_{\alpha\in\Fix(P_{A_1}^q)\setminus\{0\}}\alpha = \prod\limits_{d\in S\setminus\{4\}}  N_{\Q}^{\Q(\alpha_d)}(\alpha_d)$$
Repeating a similar computation, we also obtain
$$ \prod\limits_{\alpha\in\Fix(P_{A_1}^q)\setminus\{0\}}  Q_{A_{1}}^{k}(\alpha)=\prod\limits_{d\in S\setminus\{4\}}  N_{\Q}^{\Q(\alpha_d)}\left(Q_{A_1}^k(\alpha_d)\right).$$
The right-hand sides of the above products can be related to each other by using Theorem~\ref{thm:A1norminput} and Theorem~\ref{thm:A1normQ} together. More precisely, we have 
$$\prod\limits_{d\in S\setminus\{4\}}  N_{\Q}^{\Q(\alpha_d)}\left(Q_{A_1}^k(\alpha_d)\right)=\pm\left(\frac{k+1}{2}\right)^{2} \prod\limits_{d\in S\setminus\{4\}} N_{\Q}^{\Q(\alpha_d)}(\alpha_d).$$
The product of nonzero elements of $\Fix(P_{A_1}^q)$ is equal to $(-1)^{(q-1)/2}q-1$ by \eqref{eq:prodfixpoints}. This finishes the proof.
\end{proof}

We are ready now ready to exhibit a proof for the main result of this section

\begin{proof}[Proof of Theorem~\ref{thm:A1main}]
Suppose that $k>1$ is odd. Let $q$ be a power of an odd prime $p$. We claim that for any $p>((k+1)/2)^2+1$, the map $Q^k_{A_1}:\F_q\rightarrow\F_q$ is not a permutation. If $p\mid (k+1)$, then $Q_{A_1}^k$ does not induce a permutation. Assume otherwise, and consider the following gcd-condition
$$\gcd\left(\frac{q-1}{2}, k+1\right) \cdot \gcd\left(\frac{q+1}{2}, k+1\right) = 2.$$
If this gcd-condition does not hold, then the map $Q^k_{A_1}(x):\F_q \rightarrow \F_q$ is not a permutation since the inverse image of $\{0\}$ contains more than one element by Corollary~\ref{cor:A1invzero}. If this gcd-condition holds, then we can use Theorem~\ref{thm:A1outputprod}. Suppose that $p$ is a prime such that $p>((k+1)/2)^2+1$. Therefore
$((k+1)/2)^2\not\equiv \pm 1\pmod{p}.$
It follows by Theorem~\ref{thm:A1outputprod} that
$$\prod\limits_{\alpha\in\F_q\setminus\{0\}} Q_{A_{1}}^{k}(\alpha)\not\equiv \pm 1 \pmod{p}.$$
This computation gives the $x$-coefficient of the following polynomial
$$\prod\limits_{\alpha\in\F_q}\left(x-Q_{A_{1}}^{k}(\alpha)\right)$$
up to a plus or a minus sign. Finally, we conclude that 
$$\prod\limits_{\alpha\in\F_q} \left(x-Q_{A_{1}}^{k}(\alpha) \right)\not\equiv x^q-x\pmod p.$$
Thus, the map $Q^k_{A_1}:\F_q\rightarrow\F_q$ is not a permutation.

If $k$ is even, then we recall that $Q^k_{A_1}(x) = h(x^2)$ for some polynomial $h(x)\in\Z[x]$. It follows that the polynomial $Q^k_{A_1}(x):\F_q \rightarrow \F_q$ does not induce a permutation if $p\ ge 3$. Suppose $k>1$ is a fixed odd integer. There are only finitely many primes $p$ such that the inequality $p \leq ((k+1)/2)^2+1$ holds. We conclude that $Q_{A_1}^k$ is not exceptional for $k>1$.
\end{proof}

\section{The Main Result}
Let $\mf{g}$ be a semi-simple Lie algebra $\mf{g}$ of rank $n$. Recall that we use the notation $\Q_{\mf{g}}^k:\C^n \rightarrow \C^n$, for the generalized Chebyshev polynomials of the second kind. These polynomials are defined through exponential invariants of Bourbaki, and have integer coefficients. To our knowledge, these polynomials $Q^k_{\mf{g}}$, with $\mf{g}\neq A_1$, have not been studied in the theory of finite fields before this work. See \cite{hubertsinger2022}, and the references therein, for some applications in numerical analysis.

The classical Chebyshev polynomials of the first and second kinds, i.e. $P_{A_1}^k$ and $Q_{A_1}^k$, satisfy the same recursive relation with different initial conditions, see \eqref{eq:PQA1recursions}. This slight difference results in distinct behavior in terms of arithmetic exceptionality. The polynomials $P_{A_1}^k$ are examples of exceptional polynomials for infinitely many $k$. On the other hand, $Q_{A_1}^k$ is not exceptional for $k>1$.

In this paper, we show that there is a similar situation in the rank two case. The bivariate polynomial families $P_{A_2}^k$ and $Q_{A_2}^k$ satisfy the same recursive relation with different initial conditions. See \eqref{eq:recrelPA2}, and Theorem~\ref{thm:A2recursion}. Moreover, the polynomials $P_{A_2}^k$ are examples of exceptional polynomials for infinitely many values of $k$, see Theorem~\ref{thm:PA2expceptional}.

In this section, we will provide a criterion that concludes $Q_{A_2}^k: \F_q^2 \rightarrow \F_q^2$ is not a permutation if $p$ is sufficiently large compared to a certain function of $k$. More precisely, we will show the following.
\begin{theorem}\label{thm:A2main}
Let $k>1$ be an integer. Let $q$ be a power of an odd prime $p$. Let $\tau$ be the function defined by \eqref{eq:tau}. If $Q_{A_2}^k:\F_q^2 \rightarrow \F_q^2$ is a permutation, then $p<|\tau(k)|+1$. Moreover, $Q_{A_2}^k$ is not arithmetically exceptional.
\end{theorem}

To establish exceptionality, one needs to know if the map $Q_{A_2}^k:\F_p^2\rightarrow\F_p^2$ yields permutations for infinitely many primes $p$. In this section, we use the correspondence $\F_q \leftrightarrow  \Fix(P_{A_1}^q)$ given by \eqref{eq:one-to-one}. Our main strategy is to reduce the problem to a one-dimensional setting and focus on the diagonal of $\F_q^2$. For this purpose, we start with the following definition.
\begin{definition}\label{def:diag}
We define $f_k(x):=h_k(x,x)$, and call it the \textit{diagonal restriction} of $Q^k_{A_2}$, by choosing $y_1=y_2=x$ in Example~\ref{ex:QA2}. There is a recursive relation for $f_k$ given by Theorem~\ref{thm:A2recursion}. More precisely, $f_0=1,f_1=x$, and $f_2=x^2-x$ and
$$f_k=xf_{k-1}-xf_{k-2}+f_{k-3}, \text{ for }k\ge3.$$
\end{definition}

\begin{remark}\label{rem:diag}
Consider the diagonal $D=\{(a,a)\mathrel{|} a\in\F_q\}$ of $\F_q^2$. We have $Q^k_{A_2}(D) \subseteq D$ by \eqref{eq:QA2symmetry}. This allows us to consider the restricted map $f_k: D \rightarrow D$. If $Q^k_{A_2}$ permutes $\F_q^2$, then its diagonal restriction $f_k$ permutes $\F_q$. Equivalently, if $f_k$ does not permute $\F_q$, then $Q^k_{A_2}$ does not permute $\F_q^2$.
\end{remark}

We start with collecting some information about the polynomials $f_k$.

\begin{theorem}\label{thm:A2output012} Let $k\ge 0$ be an arbitrary integer. Then we have
	\begin{align*}
		f_k(0) &= 
		\begin{cases} 
			1 & \text{if } k \equiv 0\pmod3, \\
			0 & \text{if } k \equiv 1,2\pmod3,
		\end{cases}\\
		f_k(1) &= 
		\begin{cases} 
			1 & \text{if } k \equiv 0,1\pmod4, \\
			0 & \text{if } k \equiv 2,3\pmod4,
		\end{cases}\\
		f_k(2) &= 
		\begin{cases} 
			1 & \text{if } k \equiv 0,3\pmod6, \\
			2 & \text{if } k \equiv 1,2\pmod6, \\
			0 & \text{if } k \equiv 4,5\pmod6.
		\end{cases}
	\end{align*}
\end{theorem}
\begin{proof}
We start the proof with the first case. Putting $x=0$, we obtain $f_0(0)=1$, $f_1(0)=0$, and $f_2(0)=0^2-0=0$. Moreover
$$f_k(0)=f_{k-3}(0), \text{ for }k\ge3.$$
The result now follows easily. 

Secondly, let us consider the case with $x=1$. We have $f_0(1)=1$, $f_1(1)=1$, and $f_2(1)=1^2-1=0$. Moreover
$$f_k(1)=f_{k-1}(1)-f_{k-2}(1)+f_{k-3}(1), \text{ for }k\ge3.$$
It follows that $f_4(1)=0$ as claimed. It is straightforward to verify that the general pattern holds modulo four. 

Finally, we consider the case with $x=2$.  We have $f_0(2)=1$, $f_1(2)=2$, and $f_2(2)=2^2-2=2$. Moreover
$$f_k(2)=2f_{k-1}(2)-2f_{k-2}(2)+f_{k-3}(2), \text{ for }k\ge3.$$
It can be verified that the pattern $1,2,2,1,0,0$ keeps repeating.
\end{proof}

A direct consequence of this theorem is the elimination of certain values of $k$ from the list of permutation polynomials. More precisely, we have the following result.

\begin{corollary}\label{cor:formernotperm}
Let $q$ be a power of an odd prime. If $k\not\equiv 1,8 \pmod{12}$, then the map $f_k:\F_q\rightarrow\F_q$ is not a permutation. 
\end{corollary}
\begin{proof} Thanks to the previous theorem, we have the following table:
\[\begin{array}{|c|c|c|c|c|c|c|c|c|c|c|c|c|}\hline
	k\pmod{12} & 0 & \bf{1} & 2 & 3 & 4 & 5 & 6 & 7 & \bf{8} & 9 & 10 & 11 \\ \hline
	f_k(0)& 1 & \bf{0} & 0 & 1 & 0 & 0 & 1 & 0 & \bf{0} & 1 & 0 & 0 \\ \hline
	f_k(1)& 1 & \bf{1} & 0 & 0 & 1 & 1 & 0 & 0 & \bf{1} & 1 & 0 & 0 \\ \hline
	f_k(2)& 1 & \bf{2} & 2 & 1 & 0 & 0 & 1 & 2 & \bf{2} & 1 & 0 & 0 \\ \hline
\end{array}\]
The conclusion is now obvious.
\end{proof}

In the previous section, the values $P_{A_1}(2)=k+1$ and  $P_{A_1}(2)=(-1)^k(k+1)$ have been very important while studying arithmetic exceptionality. We aim to give an analogous result in the $A_2$-case. For this purpose, we need a better understanding of the polynomials $f_k$.

Recall that $Q^k_{A_2}(z_1,z_2)=(\chi_{k\omega_1},\chi_{k\omega_2})$. Moreover we have $z_1=y_1=S(e^{\omega_1})$ and $z_2=y_2=S(e^{\omega_2})$. One may consider formal exponential sums as complex-valued functions by putting $e^\lambda(\gamma)\mapsto e^{-2\pi i(\lambda,\gamma)}$ as in \cite[Lemma~4.1]{hoffmanwithers1988}. For this purpose, we fix some coordinates associated with the coroots. Suppose that $\gamma =u_1\alpha_1^\vee +u_2\alpha_2^\vee$. Then the components of the generalized cosine function are given by
\begin{align*}
	y_1(\gamma)&=e^{-2\pi iu_1}+e^{2\pi iu_1-2\pi iu_2}+e^{2\pi iu_2},\\
	y_2(\gamma)&=e^{2\pi iu_1}+e^{-2\pi iu_1+2\pi iu_2}+e^{-2\pi iu_2}.
\end{align*}
Note that $y_1(\gamma)$ and $y_2(\gamma)$ are complex conjugates of each other. The definition of $f_k$ requires $y_1=y_2$, and this is possible if and only if $u_1=u_2$. We set $u=u_1=u_2$, and consider $\gamma =u\alpha_1^\vee +u\alpha_2^\vee$. This allows us to focus on the cosine values forming the diagonal terms. More precisely, we have
\begin{equation}\label{eq:diag}
	y_1(\gamma)=y_2(\gamma)=1+e^{2\pi iu}+e^{-2\pi iu}=e^{-2\pi iu}\frac{1-e^{2\pi i3u}}{1-e^{2\pi iu}}.
\end{equation}

Recall that we shall plug in $u=0$, once the cancellations are done. More precisely we have $y_1(0)=3$. The choice $u_1=u_2=u$ brings a natural connection with the rank one case. The following type of elements will be used extensively in the rest of the paper
\begin{equation*}
\beta_d:=1+\alpha_d=1+\zeta_d + \zeta_d^{-1}\in 1+\Fix(P_{A_1}^q).
\end{equation*}

Our main purpose is to understand the image of the diagonal
\[D=\{ \mb{y}(u\alpha_1^\vee +u\alpha_2^\vee) \mathrel{|} u\in\Q \},\]
under the map $Q_{A_2}^k:\C^2 \rightarrow \C^2$, and its reduction modulo a certain prime ideal. Applying $Q_{A_2}^k$ to the generalized cosine function, we obtain various sums of three terms. However, summations are not well-behaved under the norm map. That's why we look for a product expression. The element $J(e^\rho)$ has a natural product decomposition where the factors are related to the positive roots \cite[VI $\S$3 Proposition 2.(i)]{bourbaki2002}. In our case, the positive roots are $\alpha_1,\alpha_2$, $\alpha_1+\alpha_2$, and we have
\begin{align*}
	J(e^{\rho})&=e^{\rho}\left(1-e^{-\alpha_1}\right)\left(1-e^{-\alpha_2}\right)\left(1-e^{-\alpha_1-\alpha_2}\right)\\
	&=e^{\rho}\left(1-e^{-2\omega_1+\omega_2}\right)\left(1-e^{\omega_1-2\omega_2}\right)\left(1-e^{-\omega_1-\omega_2}\right).
\end{align*}

Now we give an important theorem that allows us to understand the images of  $\beta_d$ under the map $f_k$. This result is analogous to the functional equation in \eqref{eq:QA1(alpha)}. 
\begin{theorem}\label{thm:fkbeta}
If the real number $u$ is not congruent to $0$ or $1/2$ modulo $\Z$, then the diagonal restriction $f_k$ of $Q^k_{A_2}$ satisfies the following functional equation
$$
f_k\left(e^{-2\pi iu}\frac{1-e^{2\pi i3u}}{1-e^{2\pi iu}} \right)=
e^{-2\pi iku}\frac{\left(1-e^{2\pi i(k+1)u}\right)\left(1-e^{2\pi i(k+2)u}\right)}{\left(1-e^{2\pi iu}\right)\left(1-e^{2\pi i2u}\right)}.
$$
\end{theorem}
\begin{proof}
Recall that $y_i=z_i=J(e^{\omega_i+\rho})/J(e^\rho)$. Moreover 
\[Q^k_{A_2}(y_1,y_2) = \left( \frac{J(e^{k\omega_1+\rho})}{J(e^\rho)}, \frac{J(e^{k\omega_2+\rho})}{J(e^\rho)} \right). \]
The denominator, namely the element $J(e^\rho)$, has a natural product decomposition as described above. Now, we focus on the numerator of the first component. If $\gamma=u\alpha_1^\vee+u\alpha_2^\vee$, then we claim that
$$
J\left(e^{(k+1)\omega_1+\omega_2}\right)(\gamma)=J\left(e^{\rho}\right)\left(\left(1+\frac{2k}{3}\right)u\alpha_1^\vee+\left(1+\frac{k}{3}\right)u\alpha_2^\vee\right).
$$
This equality can be verified by direct comparison of the six terms on both sides. We outline this computation as follows. Recall that the Cartan matrix provides a connection between the fundamental weights and simple roots. More precisely, we have
\[ \begin{array}{rcl}	
	\alpha_1 &=& 2\omega_1 - \omega_2,\\
	\alpha_2 &=& -\omega_1 +2\omega_2,\end{array}
\quad \text{ and } \quad
\begin{array}{rcl}	
	\omega_1 &=& (2/3)\alpha_1 + (1/3)\alpha_2,\\
	\omega_2 &=& (1/3)\alpha_1 + (2/3)\alpha_2.\end{array} 	
\]
The product formula requires the use of positive simple roots, and we need to make a suitable conversion. Observe that
\[\begin{bmatrix}2&-1\\-1&2\end{bmatrix}^{-1}\cdot 
\begin{bmatrix}k+1\\1\end{bmatrix} =
\begin{bmatrix}1+2k/3\\1+k/3\end{bmatrix}.
\]
Recall that we have fixed $\gamma=u\alpha_1^\vee+u\alpha_2^\vee$. This computation provides a new vector
\[\tilde\gamma = \left(1+\frac{2k}{3}\right)u\alpha_1^\vee + \left(1+\frac{k}{3}\right)u\alpha_2^\vee\]
with the property that $J\left(e^{(k+1)\omega_1+\omega_2}\right)(\gamma)=J\left(e^{\omega_1+\omega_2}\right)\left(\tilde\gamma\right).$
\end{proof}

Now, we focus on some other special cases that are not described by Theorem~\ref{thm:A2output012}. Suppose that $u$ is not congruent to $0$ or $1/2$ modulo $\Z$. Using $\gamma =u\alpha_1^\vee +u\alpha_2^\vee$ and \eqref{eq:diag}, we see that $y_1(\gamma)$ is equal to $3$ and $-1$, respectively. The images of these elements under the polynomial $f_k$ is given by the following theorem. 

\begin{corollary}\label{cor:A2outputnoundary} Let $k\ge 0$ be an arbitrary integer. Then we have
	\begin{align*}
		f_k(3)&=\frac{(k+1)(k+2)}{2}\\
		f_k(-1)&=
		\begin{cases} 
			(k+2)/2 & \text{if } k \equiv 0\pmod2, \\
			-(k+1)/2 & \text{if } k \equiv 1\pmod2.
		\end{cases}		
	\end{align*}
\end{corollary}
\begin{proof}
The values $u=0$, and $u=1/2$ correspond to $\beta_0=3$, and $\beta_2=-1$, respectively. We shall use the expressions of the previous theorem, but the cancellations must be performed before plugging in the values for $u$. In this respect, it is very similar to the computation of the boundary values as in the analogous case \eqref{eq:QA1functionaltrigo}.

To make the cancellations more transparent, we substitute $t=e^{2\pi i u}$ into the formula of Theorem~\ref{thm:fkbeta}, and obtain
$$f_k\left(t^{-1}\dfrac{1-t^3}{1-t}\right) = t^{-k}\frac{(1-t^{k+1})(1-t^{k+2})}{(1-t)(1-t^2)}.$$
The integers $k+1$ and $k+2$ come with different parity. If $k+1$ is even, then we have
$$\frac{1-t^{k+1}}{1-t^2}=1+t^2+\ldots+t^{(k+1)/2} \text{ and } \frac{1-t^{k+2}}{1-t}=1+t+\ldots+t^{k+2}.$$
Similar formulas occur if $k+2$ is even. If we put $u=0$, and therefore $t=1$, then we see that $f_k(3) = (k+1)(k+2)/2$. In the other case, we have $u=1/2$, and therefore $t=-1$. Doing a case-by-case analysis, we deduce the desired formula for $f_k(-1)$. 
\end{proof}

The values $f_k(-1)$ depend on the parity of $k$ and are not very convenient to express within more complicated formulas. We use the following formula for simplicity 
\[f_k(-1)=\frac{1+(-1)^k(2k+3)}{4}.\]
While correct, this formula masks the role played by the integers $k+1$ and $k+2$ in the values of $f_k(-1)$.

\begin{corollary}\label{cor:notperm}
Let $q$ be a power of a prime $p\ge 5$, and let $k>1$ be an integer with $k\equiv1,8\pmod{12}$. If $p\mid (k+1)(k+2)$, then $f_k:\F_q\rightarrow\F_q$ is not a permutation. 
\end{corollary}
\begin{proof}
We have $f_k(0)=0$ by Theorem~\ref{thm:A2output012}, and $f_k(3)= 0$ by Corollary~\ref{cor:A2outputnoundary}. The elements $0$ and $3$ are distinct in $\F_q$ since  $p\ge 5$. 
\end{proof}

We make one final restriction by introducing a gcd-condition as in the previous section. This result will be helpful to simplify the proof of Theorem~\ref{thm:A2normQ}. 

\begin{theorem}\label{thm:notperm}
Let $q$ be a power of a prime $p \ge 5$ and let $k>1$ be an integer with $k\equiv1,8\pmod{12}$ and $p\nmid(k+1)(k+2)$. Then the inverse image of $\{0\}$ under the map $f_{k}:\F_{q}\to\F_{q}$ contains a single element if and only if the following gcd-condition holds
	$$
	\gcd\left(\frac{(q-1)(q+1)}{3}, \frac{(k+1)(k+2)}{6}\right)=1.
	$$
\end{theorem}
\begin{proof}
Similar to the proof of Theorem~\ref{thm:invimageofzero}, we focus on the following sets
$$\mc{Z}_{k}^i := \left\{ 2\cos\left(2\pi \frac{d}{k+i}\right) \mid c\in \Z \right\}, \quad i=1,2,$$
guided by Theorem~\ref{thm:fkbeta}.

Since the values of $k$ are highly constrained, one can extract useful information about the inverse image of zero. For instance $f_k(0)=0$, $f_k(1)=1 \neq 0$, and $f_k(2)=2\neq 0$ by Theorem~\ref{thm:A2output012}. Moreover, $f_k(-1)\neq 0$ and $f_k(3)\neq 0$ by Corollary~\ref{cor:A2outputnoundary}. We have the following table of values
\[\begin{array}{|c|c|c|c|c|c|c|c|c|c|c|c|c|}\hline
	d & 1 & 6 & 4 & 3 & 2 \\ \hline
	u & 0 & 1/6 & 1/4 & 1/3 & 1/2 \\ \hline
	\alpha_d = 2\cos(2\pi u) & 2 & 1 & 0 & -1 & -2 \\ \hline
	\beta_d = 1+2\cos(2\pi u) & 3 & 2 & 1 & 0 & -1 \\ \hline
\end{array}\]

Recall the one-to-one correspondence \eqref{eq:one-to-one}. 
Under the hypotheses of the theorem, we shall show that gcd-condition of the theorem is equivalent to the following
$$(\mc{A}_{q} \cup \mc{B}_{q}) \cap (\mc{Z}_{k}^1 \cup \mc{Z}_{k}^2) = \{-1\}.$$
Suppose that this intersection contains an element $2\cos(2\pi u)$, other than $\alpha_3=-1$, where $u\in(0,1/2)$ is a unique rational number. If $u=a/b$ is written in its lowest terms, then the denominator $b$ cannot be $1,2,3,4$, or $6$ by the above table. We further write, for some unique integers $c$ and $d$, that
\[u =\frac{a}{b}=\frac{c}{q\pm1} = \frac{d}{k+i}, \quad i=1,2,\]
using the definitions of the sets within the above intersection. 

If $b$ is divisible by a prime $l \ge 5$, then the prime $l$ divides both $q\pm1$ and $k+i$ for some choice of a plus or minus sign and $i \in\{1,2\}$. It follows that the prime $l$ divides the expression $\gcd((q-1)(q+1),(k+1)(k+2))$. This implies that the gcd-condition of the theorem fails.  

It remains to consider the cases that $b$ is divisible by $8$, $9$, or  $12$ according to the above discussion. 

Suppose that $9 \mid b$. Then $\gcd((q-1)(q+1),(k+1)(k+2))$ is divisible by $9$. This implies that the gcd-condition of the theorem fails. Indeed, we have $f_k(\beta_9) =0$ if and only if $9 \mid k+1$ or $9 \mid k+2$. Moreover $\beta_9 \pmod{\mf{p}}$ is an element of $\F_q$ if and  only if $9 \mid q-1$ or $9 \mid q+1$. 

The other cases, namely $8\mid b$ and $12\mid b$, do not occur. Recall that $k\equiv1,8\pmod{12}$. It can be verified that $f_k(\beta_8)$ and $f_k(\beta_{12})$ are never zero modulo $\mf{p}$. 
\end{proof}

Now we shall focus on the norms of the elements $\beta_d=1 + \zeta_d + \zeta_d^{-1}$ and $f_k(\beta)$. We generalize the method in the previous section. Recall that $\Z^{+}$ is the set of positive integers, and $\mathbb{P}$ is its subset consisting of prime numbers. We start with the following theorem, which is analogous to Theorem~\ref{thm:A1norminput}. 

\begin{theorem}\label{thm:A2norminput}
The norm of $\beta_d$ is given by:
\[N_{\Q}^{\Q(\beta_d)}(\beta_d) = 
\begin{cases} 
	3 & \text{if } d = 1, \\
	0 & \text{if } d = 3, \\
	\pm l & \text{if } d = 3l^a \text{ for some }l\in\mathbb{P}, a\in\Z^{+},\\
	\pm 1 & \text{otherwise}.
\end{cases}\]
\end{theorem}
\begin{proof}
The cases where $d$ equals 1 or 3 are trivial as $\beta_1=3$ and $\beta_3=0$. For the remaining cases, we shall prove
\[N_{\Q}^{\Q(\zeta_d)}(\beta_d) = 
\begin{cases} 
	l^2 & \text{if } d = 3l^a \text{ for some }l\in\mathbb{P}, a\in\Z^{+},\\
	1 & \text{otherwise}.
\end{cases}\]
in light of (\ref{eq:norminextension}). We note the identity 
$$\beta_d=1 + \zeta_d + \zeta_d^{-1}=\zeta_d^{-1}\left(\dfrac{1-\zeta_d^3}{1-\zeta_d}\right)$$ 
Using the multiplicativity of the norm and the fact that $\zeta_d^{-1}$ is a unit, we have the following:
$$N_{\Q}^{\cycl{d}}(\beta_d) = \pm \frac{N_{\Q}^{\cycl{d}} \left(1-\zeta_d^3\right)}{N_{\Q}^{\cycl{d}}\left(1-\zeta_d\right)}.$$
We want to utilize Proposition~\ref{prop:norm}. We analyze the numerator and denominator for each case of $d$.

If $3\nmid d$, then the numerator $1-\zeta_d^3$ and denominator $1-\zeta_d$ have the same norm by Proposition~\ref{prop:norm}. Hence, $N_{\Q}^{\cycl{d}}(\beta_d)=1$.
	
If $d=3m$ with $m>1$ not a prime power, then the numerator $1-\zeta_d^3=1-\zeta_m$ and the denominator $1-\zeta_d$ are units by Proposition~\ref{prop:norm}. Hence, $N_{\Q}^{\cycl{d}}(\beta_d)=1$.
	
If $d=3^a$ with $a>1$, the numerator is $1-\zeta_{3^a}^3 = 1-\zeta_{3^{a-1}}$ and the denominator is $1-\zeta_{3^a}$. By Proposition~\ref{prop:norm}, both have norm $3$ in the minimal fields containing them. Moreover, we have the following relative degree
$$[\cycl{3^a}:\cycl{3^{a-1}}] = \varphi(3^a)/\varphi(3^{a-1}) = 3^{a-1}/3^{a-2} = 3$$
Therefore, we have $N_{\Q}^{\cycl{d}}(\beta) = 3^3/3 = 3^2$ by using relative norms.
	
If $d=3l^a$ where $l\neq3$ is a prime and $a\geq1$, the numerator is $1-\zeta_{3l^a}^3 = 1-\zeta_{l^a}$ and the denominator is $1-\zeta_{3l^a}$. By Proposition~\ref{prop:norm}, the numerator has norm $l$ in the field $\cycl{l^a}$, and the denominator is a unit in the field $\cycl{3l^a}$. Moreover, we have the following relative degree
$$[\cycl{3l^a}:\cycl{l^a}] = \varphi(3l^a)/\varphi(l^a) = \varphi(3) = 2.$$
Therefore, we have $N_{\Q}^{\cycl{d}}(\beta) = l^2/1 = l^2$ by using relative norms.
\end{proof}

The next theorem follows the ideas introduced in the previous section, see Theorem~\ref{thm:A1normQ}. For simplicity, we consider only the remaining cases of $k$ that may possibly produce permutations of $\F_q$. A notable difference from the analogous result is that the gcd-condition depends on $k+2$ as well as $k+1$.

\begin{theorem}\label{thm:A2normQ}
Let $q$ be a power of a prime $p\ge 5$ and let $k>1$ be an integer with $k\equiv1,8\pmod{12}$. Suppose further that $p\nmid(k+1)(k+2)$ and
$$\gcd\left(\frac{(q-1)(q+1)}{3}, \frac{(k+1)(k+2)}{6}\right)=1.$$
If $d$ is a divisor of $q-1$ or $q+1$, then the norm of $f_k(\beta_d)$ is given by:
$$
N_{\Q}^{\Q(\beta_d)}\left(f_k(\beta_d)\right) = 
\begin{cases}
	\dfrac{(k+1)(k+2)}{2} & \text{if } d = 1, \\
	\dfrac{1+(-1)^k(2k+3)}{4} & \text{if } d = 2, \\
	0 & \text{if } d = 3, \\
	\pm l & \text{if } d = 3l^a \text{ for some }l\in\mathbb{P}, a\in\Z^{+},\\
	\pm 1 & \text{otherwise}.
\end{cases}
$$
\end{theorem}
\begin{proof}
The cases where $d$ equals 1 or 2 are obtained by Corollary~\ref{cor:A2outputnoundary}. The case $d=3$ is obtained by Theorem~\ref{thm:A2output012}. For the remaining cases, we shall prove
\[N_{\Q}^{\Q(\zeta_d)}\left(f_k(\beta_d)\right) = 
	\begin{cases} 
		l^2 & \text{if } d = 3l^a \text{ for some }l\in\mathbb{P}, a\in\Z^{+},\\
		1 & \text{otherwise}.
	\end{cases}
	\]
in light of \eqref{eq:norminextension}. Using Theorem~\ref{thm:fkbeta}, we obtain the identity 
$$f_k(\beta_d)=\zeta_d^{-k}\dfrac{\left(1-\zeta_d^{k+1}\right)\left(1-\zeta_d^{k+2}\right)}{\left(1-\zeta_d\right)\left(1-\zeta_d^2\right)}$$
In order to compute the norm of this element, we want to utilize Proposition~\ref{prop:norm}. For this purpose, we consider various cases of $d$.

Suppose that $3\nmid d$. For some $i$ and $j$, with $\{i,j\}=\{1,2\}$, we assume that $k+i$ is odd and $k+j$ is even. Note that $k+j$ is not divisible by $4$ since $k\equiv1,8\pmod{12}$. The gcd-condition now implies that the following conjugacy relations hold: $$1-\zeta_d \sim 1-\zeta_d^{k+i}\quad \text{ and }\quad 1-\zeta_d^2 \sim 1-\zeta_d^{k+j}.$$ 
Hence, we conclude that $N_{\Q}^{\Q(\zeta_d)}\left(f_k(\beta_d)\right)=1$.

Suppose that $d=3m>3$ where $m$ is not divisible by three, and $m$ is not a prime power. The terms in the denominator, namely $1-\zeta_d$ and $1-\zeta_d^2$ are units by Proposition~\ref{prop:norm}. Now, let us focus on the terms in the numerator, namely $1-\zeta_d^{k+1}$ and $1-\zeta_d^{k+2}$. We claim that they are units as well. Assume otherwise, and suppose that  $1-\zeta_d^{k+i}$ is not a unit for some $i=1,2$. Then this means that at least two prime factors are canceled from $d=3m$. Recall that $d$ is a divisor of $q-1$ or $q+1$, and the gcd-condition puts a severe restriction on $k+1$ and $k+2$ values. We must have $6 \mid \gcd(d,k+i)$. However, $k+i$, for some $i=1,2$, is not divisible by $6$ since $k\equiv1,8\pmod{12}$. Hence, we conclude that $N_{\Q}^{\Q(\zeta_d)}\left(f_k(\beta_d)\right)=1$.

Suppose that $d=3^a$ with $a \ge 2$. The integer $k$ is not divisible by three since  $k\equiv1,8\pmod{12}$. For some $i$ and $j$, with $\{i,j\}=\{1,2\}$, we assume that $k+i$ is divisible by there and $k+j$ is not divisible by three. The gcd-condition implies that the element $1-\zeta_d^{k+i}$ is conjugate to $1-\zeta_d^3$. Note that both $1-\zeta_d$ and $1-\zeta_d^3 = 1-\zeta_{3^{a-1}}$ have norm $3$ in the minimal fields containing them by Proposition~\ref{prop:norm}. Moreover, we have the following relative degree
$$[\cycl{3^a}:\cycl{3^{a-1}}] = \varphi(3^a)/\varphi(3^{a-1}) = 3^{a-1}/3^{a-2} = 3$$
On the other hand, the elements
$1-\zeta_d^2$ and $1-\zeta_d^{k+j}$ are conjugates of each other, and their contribution to the norm computation is trivial. Therefore, we have $N_{\Q}^{\Q(\zeta_d)}\left(f_k(\beta_d)\right) = 3^3/3 = 3^2$ by using relative norms.

Finally, suppose that $d=3l^a$ where $l\neq3$ is a prime and $a\geq1$. The term $1-\zeta_d$ in the denominator is a unit by Proposition~\ref{prop:norm}. For some $i$ and $j$, with $\{i,j\}=\{1,2\}$, we assume that $k+i$ is divisible by there and $k+j$ is not divisible by three. Since $k\equiv1,8\pmod{12}$, the integer $k+i$ is odd and $k+j$ is even. We have $1-\zeta_d^{k+i} = 1-\zeta_{l^a}^{(k+i)/3}$. The gcd-condition implies that $1-\zeta_d^{k+i}$ is a conjugate of $1-\zeta_{l^a}$. Here, it is important that $k+i$ is odd. By Proposition~\ref{prop:norm}, the element $1-\zeta_{l^a}$  has norm $l$ in the field $\cycl{l^a}$. Moreover, we have the following relative degree
$$[\cycl{3l^a}:\cycl{l^a}] = \varphi(3l^a)/\varphi(l^a) = \varphi(3) = 2.$$
Recall that $k+j$ is not divisible by three. We observe that $k+j$ is even but not divisible by four. It follows that the elements $1-\zeta_d^2$ and $1-\zeta_d^{k+j}$ are conjugates of each other, and their contribution to the norm computation is trivial. We finally conclude that $N_{\Q}^{\Q(\zeta_d)}\left(f_k(\beta_d)\right) = l^2/1 = l^2$ by using relative norms.
\end{proof}

The following theorem is the last step before we prove our main result. To ease the notation, we define the following function
\begin{equation}\label{eq:tau}
	\tau(k)=\frac{(k+1)(k+2)\left(1+(-1)^k(2k+3)\right)}{24}.
\end{equation}
For a fixed value of $k$, the function $|\tau(k)|+1$ gives an upper bound on primes $p$ for which our main result, see Theorem~\ref{thm:A2main}, is inconclusive. 

\begin{theorem}\label{thm:A2outputprod}
Let $q$ be a power of a prime $p\ge 5$ and let $k>1$ be an integer with $k\equiv1,8\pmod{12}$. Suppose further that $p\nmid(k+1)(k+2)$ and
$$\gcd\left(\frac{(q-1)(q+1)}{3}, \frac{(k+1)(k+2)}{6}\right)=1.$$	
Let $\mc{F}=1+\Fix(P_{A_1}^q)$ be the set obtained by translating $\Fix(P_{A_1}^q)$ by one. Then we have
$$
\prod\limits_{\beta\in \mc{F} \setminus\{0\}} f_{k}(\beta)= \pm\tau(k) (q\pm1).
$$
Moreover, the plus or minus sign within the last term is uniquely determined by the condition $6|(q\pm1)$.
\end{theorem}
\begin{proof}
Any element of $\mc{F} = 1+\Fix(P_{A_1}^q)$ is a conjugate of $\beta_d=1+\alpha_d$ for some divisor $d$ of $q-1$ or $q+1$. Moreover, all conjugates of $\beta_d$ are included in the number field $\Q(\alpha_d) = \Q(\beta_d)$. This is a direct consequence of the one-to-one correspondence given by \eqref{eq:one-to-one}. 

Let $S$ be the subset of positive integers $1 \leq d \leq q+1$ such that $\beta_d \in 1+\Fix(P_{A_1}^q)$. The set $S$ consists of divisors of $q-1$ or $q+1$ by \eqref{eq:one-to-one}. We have the following partition induced by the conjugacy equivalence relation 
$$ \mc{F}=1+\Fix(P_{A_1}^q) = \bigcup_{d\in S} [\beta_d]$$
The equivalence class of $\beta_3=0$ has a single element since $\beta_3\in\Q$. Symbolically, we have $[\beta_3] = \{0\}$. It follows that 
$$\prod\limits_{\beta\in \mc{F}\setminus\{0\}}\beta = \prod\limits_{d\in S\setminus\{3\}}  N_{\Q}^{\Q(\beta_d)}\left(\beta_d\right).$$
Repeating a similar computation, we also obtain
$$ \prod\limits_{\beta\in\mc{F}\setminus\{0\}}  f_k(\beta)=\prod\limits_{d\in S\setminus\{3\}}  N_{\Q}^{\Q(\beta_d)}\left(f_k(\beta_d)\right).$$
The right-hand sides of the above products can be related to each other by using Theorem~\ref{thm:A2norminput}, and Theorem~\ref{thm:A2normQ} together. More precisely, we have 
$$\prod\limits_{d\in S\setminus\{3\}}  N_{\Q}^{\Q(\beta_d)}\left(f_k(\beta_d)\right)=\pm\tau(k) \prod\limits_{d\in S\setminus\{3\}} N_{\Q}^{\Q(\beta_d)}\left(\beta_d\right).$$
If $p\neq 3$, then we observe that $P_{A_1}^q(x-1)=x^q+\ldots \pm q x-1$. 
The product of nonzero elements of $\mc{F}$ is equal to $\pm q\pm1$. Obviously, the term $q \pm 1 $ is  divisible by two. The divisibility by three follows from Remark~\ref{rm:coef} together with Theorem~\ref{thm:A2norminput}.
\end{proof}

We are ready now ready to exhibit a proof for the main result of this paper.

\begin{proof}[Proof of Theorem~\ref{thm:A2main}]
Let $q$ be a power of a prime $p\ge 5$ and let $k>1$ be an integer. If $k \not \equiv 1,8\pmod{12}$, then the map $f_k:\F_q\rightarrow\F_q$ is not a permutation by Corollary~\ref{cor:formernotperm}. If $k \equiv 1,8\pmod{12}$ but $p \mid (k+1)(k+1)$, then the map $f_k:\F_q\rightarrow\F_q$ is not a permutation by Corollary~\ref{cor:notperm}. Suppose otherwise, and consider the gcd-condition of Theorem~\ref{thm:notperm} given below
$$
\gcd\left(\frac{(q-1)(q+1)}{3}, \frac{(k+1)(k+2)}{6}\right)=1.
$$
If this gcd-condition does not hold, then the map $f_k:\F_q \rightarrow \F_q$ is not a permutation since the inverse image of $\{0\}$ contains more than one element. 

It remains to consider the case where $k\equiv 1,8 \pmod{12}$, $p \nmid (k+1)(k+1)$, and the gcd-condition above holds. Under these hypotheses, we can use Theorem~\ref{thm:A2outputprod}. We have
$$
\prod\limits_{\beta\in \mc{F}\setminus\{0\}} f_{k}(\beta)= \pm\tau(k) (q\pm1).
$$
This computation gives the $x$-coefficient of the following polynomial
$$\prod\limits_{\beta\in\F_q}\left(x-f_k(\beta)\right)$$
up to a plus or a minus sign. Suppose that $p>|\tau(k)|+1$. Finally, we see that 
$$\prod\limits_{\beta\in\F_q} \left(x-f_k(\beta) \right)\not\equiv x^q-x\pmod p.$$
Thus, the map $f_k:\F_q\rightarrow\F_q$ is not a permutation. We also conclude that the map $Q_{A_2}^k:\F_q^2\rightarrow\F_q^2$ is not a permutation by following Remark~\ref{rem:diag}.

For a fixed value of $k$, there are only finitely many primes $p$ such that the inequality $p \leq |\tau(k)|+1$ holds. We conclude that $Q_{A_2}^k$ is not exceptional for $k>1$.
\end{proof}

\begin{remark} We hope that the ideas of this work can be further generalized to other second kind generalized Chebyshev polynomials associated with arbitrary semi-simple Lie algebras. After completing the study of $A_1$ and $A_2$, one may naturally turn to the next simplest cases, namely $B_2$ and $G_2$. However, there is an immediate difficulty. In those cases, the property \eqref{eq:QA2symmetry} is no longer true. In other words, at first glance, restricting the functions $Q_{B_2}^k$ or $Q_{G_2}^k$ to the diagonal of $\F_q^2$ does not appear to be a useful approach.
\end{remark}

\section{Acknowledgement}
The authors gratefully acknowledge the financial support provided by the Scientific and Technological Research Council of Türkiye (TÜBİTAK) through Project No. 124F146.

\bibliographystyle{abbrv}
\bibliography{references}

\end{document}